\documentclass[11pt,a4paper,leqno]{amsart}
\usepackage{amsmath,amsfonts,amssymb,amsthm}
\usepackage[alphabetic]{amsrefs}
\usepackage{hyperref, url}
\usepackage{graphicx, color}
\usepackage{enumerate}
\usepackage{fullpage}
\usepackage{xspace}
\newtheorem{lemma}{Lemma}
\newtheorem{proposition}{Proposition}
\newtheorem{theorem}{Theorem}
\newtheorem{corollary}{Corollary}

\theoremstyle{definition}
\newtheorem{example}{Example}
\newtheorem{remark}{Remark}

\newtheorem{definition}{Definition}
\newtheorem*{rexample}{Running example}
\newtheorem{conjecture}{Conjecture}
\newtheorem{claim}{Claim}

\newcommand{\covectors}{\ensuremath{\mathcal{L}}}
\newcommand{\cocircuits}{\ensuremath{\mathcal{C}^*}}

\newcommand{\rk}{\ensuremath{{\rm rank}}}
\newcommand{\uX}{\ensuremath{\underline{\mathcal{X}}}\xspace}
\newcommand{\oX}{\ensuremath{\overline{\mathcal{X}}}\xspace}
\newcommand{\ac}{\ensuremath{{\rm amp}}}
\newcommand{\product}{\square}

\newcommand{\C}{\mathrm C}
\newcommand{\F}{\mathrm F}
\newcommand{\Q}{\mathrm Q}

\DeclareMathOperator{\pr}{\rm pr}

\newcommand{\amp}{\ensuremath{{\rm amp}}}

\renewcommand{\SS}{\ensuremath{\mathcal{S}}}
\newcommand{\uparr}{\ensuremath{\hspace{3pt}\uparrow\hspace{-2pt}}}

\DeclareMathOperator{\vcd}{VC-dim}
\renewcommand{\SS}{\ensuremath{\mathcal{S}}}

\definecolor{color1}{RGB}{255,128,0}

\begin{document}


\centerline{\bf\Large Ample completions of  oriented matroids}
\centerline{\bf\Large and complexes of uniform oriented matroids}

\bigskip
\centerline{\sc \large Victor Chepoi$^{1}$, Kolja Knauer$^{1,2}$, and Manon Philibert$^{1}$}

\bigskip
\centerline{$^{1}$LIS, Aix-Marseille Universit\'e, CNRS, and Universit\'e de Toulon}
\centerline{Facult\'e des Sciences de Luminy, F-13288 Marseille Cedex 9, France}
\centerline{ {\sf \{victor.chepoi,kolja.knauer,manon.philibert\}@lis-lab.fr} }

\bigskip
\centerline{$^{2}$Departament de Matem\`atiques i Inform\`atica, Universitat de Barcelona (UB),}
\centerline{Barcelona, Spain}

\bigskip\medskip
\noindent
{\footnotesize {\bf Abstract} This paper considers completions of tope graphs of COMs
(complexes of oriented matroids) to ample partial cubes of the same
VC-dimension. We show that these exist for OMs (oriented matroids) and CUOMs
(complexes of uniform oriented matroids). This implies that tope graphs of OMs and CUOMs
satisfy the sample compression conjecture -- one of the central open questions of learning theory.
We conjecture that the tope graph of every COM can be completed to an ample partial cube without
increasing the VC-dimension.}


\section{Introduction}
Oriented matroids (OMs), co-invented by Bland and Las Vergnas~\cite{BlLV} and
Folkman and Lawrence~\cite{FoLa}, represent a unified combinatorial theory of
orientations of  ordinary matroids. They capture the basic properties of sign
vectors representing the circuits in a directed graph and the regions in a
central hyperplane arrangement in  $\mathbb{R}^m$. Oriented matroids are
systems of sign vectors 
satisfying three simple axioms (composition, strong elimination, and symmetry)
and may be defined
in a multitude of 
ways, see the book by Bj\"orner et al.~\cite{BjLVStWhZi}. The tope graphs of
OMs 
can be viewed as 
subgraphs of the hypercube $\Q_m$ 
satisfying  two strong properties: they are centrally-symmetric and are
isometric subgraphs of $\Q_m$, i.e., are antipodal partial
cubes~\cite{BjLVStWhZi}.

Ample set systems (AMPs) have been introduced by Lawrence \cite{La} as asymmetric
counterparts  of  oriented  matroids and have been re-discovered independently
by several works in different contexts~\cite{BaChDrKo,BoRa,Wiedemann}.
Consequently, they received different names: lopsided~\cite{La},
simple~\cite{Wiedemann}, extremal~\cite{BoRa}, and ample~\cite{BaChDrKo,Dr}.
Lawrence~\cite{La} defined ample set systems for the investigation of the possible
sign patterns realized by a convex set in $\mathbb{R}^m$. Ample sets admit a
multitude of combinatorial and geometric characterizations
\cite{BaChDrKo,BoRa,La} and comprise many natural examples arising from
discrete geometry, combinatorics, 
and geometry of groups \cite{BaChDrKo,La}.
Analogously to tope graphs of OMs, AMPs induce isometric subgraphs of $\Q_m$.
In fact, they satisfy a stronger property: any two parallel cubes are
connected in the set by  a shortest path of parallel cubes.

Complexes of oriented matroids (COMs) have been introduced and investigated in
\cite{BaChKn} as a far-reaching natural common generalization of oriented
matroids and ample set systems. COMs are defined in a similar way as OMs, simply
replacing the global axiom of symmetry by a local axiom of face symmetry. This
simple alteration leads to a rich combinatorial and geometric structure that is
build from OM faces but is quite different from OMs. Replacing each face by a
PL-ball, each COM leads to a contractible cell complex (topologically, OMs are
spheres and AMPs are contractible cubical complexes). The tope graphs of COMs
are still isometric subgraphs of hypercubes; as such, they have been
characterized  in \cite{KnMa}.

Set families are fundamental objects in combinatorics, algorithmics, machine
learning, discrete geometry, and combinatorial optimization. The
Vapnik-Chervonenkis dimension (\emph{VC-dimension} for short)  of a set family
was introduced by  Vapnik and Chervonenkis \cite{VaCh} and plays  a central
role in the theory of PAC-learning. The VC-dimension was adopted in the above
areas as a complexity measure and as a combinatorial dimension of the set
family. The topes of OMs,  COMs, and AMPs (viewed as isometric subgraphs of
hypercubes) give rise to set families for which the VC-dimension has a
particular significance: the VC-dimension of an AMP is the largest dimension of
a cube of its cube complex, the VC-dimension of an OM is its rank, and the
VC-dimension of a COM is the largest VC-dimension of its faces.

Littlestone and Warmuth~\cite{LiWa} introduced the sample compression technique
for deriving generalization bounds in machine learning. Floyd and
Warmuth~\cite{FlWa} asked whether any set family of VC-dimension $d$ has a
sample compression scheme of size~$O(d)$. This question, known as the {\it 
sample compression conjecture}, remains one of the
oldest open problems in machine learning. It was shown in \cite{MoYe} that
labeled compression schemes of size $O(2^d)$ exist. Moran and Warmuth
\cite{MoWa} designed labeled compression schemes of size $d$ for ample set systems.
Chalopin et al. \cite{ChChMoWa} designed (stronger) unlabeled compression
schemes of size $d$ for maximum families and characterized such schemes for
ample set systems.
In view of the above, it was noticed in \cite{RuRuBa} and \cite{MoWa} that  the
sample compression conjecture would be solved if {\it any
set family of VC-dimension $d$ can be completed to an ample (or maximum) set system of
VC-dimension $O(d)$} or {\it covered by $O(2^d)$ ample set systems of VC-dimension
$O(d)$.}

This opens a perspective that apart from its application to sample compression,
is interesting in its own right: ample completions of structured set families. This is
extending a given set system to an ample system by adding sets.
A natural problem is ample completions of set families defined by \emph{partial
cubes} (i.e., isometric subgraphs of hypercubes). In~\cite{ChKnPh}, we prove
that any partial cube of VC-dimension 2 admits an ample completion of
VC-dimension 2. Moreover, we give a set family of VC-dimension 2 which has no
ample completion of the same VC-dimension. In the present paper, we give an
example of a partial cube of VC-dimension 3 which cannot be completed to an
ample set system of VC-dimension 3. Hence, in higher dimension we cannot complete all
partial cubes without increasing the VC-dimension. In the light of the above
perspective, one may ask if \emph{there exists a constant $c$ such that every
partial cube of VC-dimension $d$ admits an ample completion of VC-dimension
$\le cd$?} Even stronger, we wonder if \emph{partial cubes of VC-dimension $d$
admit an ample completion of VC-dimension $d+c$}. Note that no such additive
constant $c$ exists for general set families~\cite{RuRuBa}. In \cite{ChKnPh},
we perform the ample completion of a partial cube of VC-dimension $2$ in two
steps. First, we show that they can be completed to tope graphs of COMs and
next we complete the resulting graphs into ample ones, in both cases,
without increasing the VC-dimension. In the present article, we are 
interested in the completion of this intermediate class of partial cubes. For
COMs, we are inclined to
believe that the following stronger result holds:

\begin{conjecture} \label{conjectureCOM}
	The tope graph of any COM of VC-dimension $d$ has an ample completion of
	VC-dimension $d$.
\end{conjecture}

COMs of rank $2$ have the nice property that their faces are uniform OMs.
This is not longer true in higher dimensions: COMs whose faces are uniform OMs
constitute a proper subclass of COMs (that we call CUOMs). In this paper, we prove that
Conjecture \ref{conjectureCOM} holds for all tope graphs of OMs and CUOMs. This
proves that set families arising from topes of OMs and CUOMs satisfy the sample
compression conjecture. In Fig. \ref{fig:running_example_CUOM}, we present an
example of the tope graph of a CUOM of VC-dimension
$3$, which we further use as a running example.

\begin{figure}[htb]
	\centering
	\includegraphics[width=.35\textwidth]{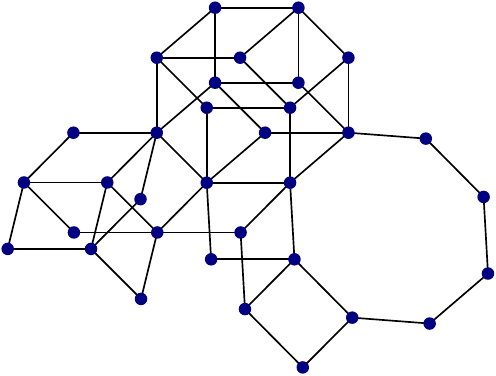}
	\caption{The tope graph $M$ of a CUOM}
	\label{fig:running_example_CUOM}
\end{figure}

%
%

\section{Preliminaries}

\subsection{VC-dimension}
Let $\SS$ be a family of subsets of an $m$-element  set $U$.  A subset $X$ of
$U$ is \emph{shattered} by $\SS$ if for all $Y\subseteq X$ there exists
$S\in\SS$ such that $S\cap X=Y$. The \emph{Vapnik-Chervonenkis dimension}
\cite{VaCh} (the \emph{VC-dimension} for short)  $\vcd(\SS)$ of $\SS$ is the
cardinality of the largest subset of $U$ shattered by $\SS$. Any set family
$\SS\subseteq 2^U$ can be viewed as a subset of vertices of the $m$-dimensional
hypercube $\Q_m=\Q(U)$. Denote by $G(\SS)$ the subgraph of $\Q_m$ induced by
the vertices of $\Q_m$ corresponding to the sets of $\SS$; $G(\SS)$ is called
the \emph{1-inclusion graph} of $\SS$; each subgraph of  $\Q_m$ is the
1-inclusion graph of a family of subsets of $U$. A subgraph $G$ of $\Q_m$ has
VC-dimension $d$ if $G$ is the 1-inclusion graph of a set family of
VC-dimension $d$. For a subgraph $G$ of $\Q_m$ we denote by $\C(G)$ the
smallest cube of $\Q_m$ containing $G$.

An {\it $X$-cube} of $\Q_m$ is the 1-inclusion graph of the set family  $\{
Y\cup X': X'\subseteq X\}$, where $Y$ is a subset of $U\setminus X$. If
$|X|=m'$, then any $X$-cube is an $m'$-dimensional subcube of $\Q_m$ and $\Q_m$
contains $2^{m-m'}$ $X$-cubes. We call any two $X$-cubes \emph{parallel cubes}.
A subset $X$ of $U$ is \emph{strongly shattered} by $\SS$ if the 1-inclusion
graph $G(\SS)$ of $\SS$ contains an $X$-cube. Denote by $\oX(\SS)$ and
$\uX(\SS)$ the families  consisting of all shattered and of all strongly
shattered sets of $\SS$, respectively. Clearly, $\uX(\SS)\subseteq \oX(\SS)$
and both $\oX(\SS)$ and $\uX(\SS)$ are closed under taking subsets, i.e.,
$\oX(\SS)$ and $\uX(\SS)$ are \emph{abstract simplicial complexes}. The
VC-dimension $\vcd(\SS)$ of $\SS$ is thus the size of a largest set shattered
by $\SS$, i.e., the dimension of the simplicial complex $\oX(\SS)$.

Two important inequalities relate a set family $\SS\subseteq 2^{U}$ with its
VC-dimension. The first one, the \emph{Sauer-Shelah lemma} \cite{Sauer,Shelah}
establishes that if $|U|=m$, then the number of sets in a set family
$\SS\subseteq 2^{U}$ with VC-dimension  $d$ is upper bounded by
$\Phi_d(m):=\sum_{i=0}^d \binom{m}{i}$. The second stronger inequality, called
the \emph{sandwich lemma} \cite{AnRoSa,BoRa,Dr,Pa}, proves that $|\SS|$ is
sandwiched between the number of strongly shattered sets and the number of
shattered sets, i.e., $\lvert\uX(\SS)\rvert \le \lvert
\SS\rvert \le \lvert \oX(\SS)\rvert$. If $d=\vcd(\SS)$ and $m=\lvert U\rvert$,
then $\oX(\SS)$ cannot contain more than $\Phi_d(m)$ simplices, thus the
sandwich lemma yields the Sauer-Shelah lemma. The set families for which the
Sauer-Shelah bounds are tight are called {\it maximum sets}  \cite{GaWe,FlWa}
and the set families for which the upper bounds in the sandwich lemma are tight
are called {\it ample, lopsided, and extremal sets} \cite{BaChDrKo,BoRa,La}.
Every maximum set system is ample, but not vice versa.

\subsection{Partial cubes}
All graphs $G=(V,E)$ in this paper  are finite, connected,
and simple. The {\it distance} $d(u,v):=d_G(u,v)$ between two vertices $u$ and
$v$ is the length of a shortest $(u,v)$-path, and the {\it interval} $I(u,v)$
between $u$ and $v$ consists of all vertices on shortest $(u,v)$-paths:
$I(u,v):=\{ x\in V: d(u,x)+d(x,v)=d(u,v)\}.$ An induced subgraph $H$ of $G$ is
{\it isometric} if the distance between any pair of vertices in $H$ is
the same as that in $G$. An induced subgraph $H$ of $G$ (or its vertex set $S$)
is called {\it convex} if it includes the interval of $G$ between
any two vertices of $H$.
A subset $S\subseteq V$ or the subgraph $H$ of $G$ induced by $S$ is called
{\it gated} (in $G$)~\cite{DrSch} if for every vertex $x$ outside $H$ there
exists a vertex $x'$ (the {\it gate} of $x$) in $H$ such that each vertex $y$
of $H$ is connected with $x$ by a shortest path passing through the gate $x'$.
It is easy to see that if $x$ has a gate in $H$, then it is unique and that
gated sets are convex.

A graph $G=(V,E)$ is {\it isometrically embeddable} into a graph $H=(W,F)$ if
there exists a mapping $\varphi : V\rightarrow W$ such that $d_H(\varphi
(u),\varphi (v))=d_G(u,v)$ for all vertices $u,v\in V$, i.e., $\varphi(G)$ is
an isometric subgraph of $H$. A graph $G$ is called a {\it partial cube} if it
admits an isometric embedding into some hypercube $\Q_m$. For an edge $uv$ of
$G$, let $W(u,v)=\{ x\in V: d(x,u)<d(x,v)\}$. For an edge $uv$, the sets
$W(u,v)$ and $W(v,u)$ are called {\it complementary halfspaces} of $G$.

\begin{theorem} \cite{Dj} \label{Djokovic}
	A graph $G$ is a partial cube if and only if $G$ is bipartite and for any
	edge $uv$the sets $W(u,v)$ and $W(v,u)$ are convex.
\end{theorem}

Djokovi\'{c}~\cite{Dj} introduced the following binary relation $\Theta$  on
the edges of $G$:  for two edges $f=uv$ and $f'=u'v'$, we set $f\Theta f'$ if
and only if $u'\in W(u,v)$ and $v'\in W(v,u)$. Under the conditions of the
theorem, $f\Theta f'$ if and only if $W(u,v)=W(u',v')$ and $W(v,u)=W(v',u')$,
i.e. $\Theta$ is an equivalence relation. Let $E_1,\ldots,E_m$ be the
equivalence classes of $\Theta$ and let $b$ be an arbitrary vertex taken as the
basepoint of $G$. For a $\Theta$-class $E_i$, let $\{ G^-_i,G^+_i\}$ be the
pair of complementary halfspaces of $G$ defined by setting $G^-_i:=G(W(u,v))$
and $G^+_i:=G(W(v,u))$ for an arbitrary edge $uv\in E_i$ such that $b\in
G^-_i$. 
The isometric embedding $\varphi$ of $G$ into the $m$-dimensional hypercube
$\Q_m$ is obtained by setting $\varphi(v):=\{ i: v\in G^+_i\}$ for any vertex
$v\in V$.

\begin{rexample} The complementary halfspaces $M^-_i$ and $M^+_i$ of
the running example $M$ defined by the $\Theta$-class $E_i$ are illustrated in
Fig. \ref{fig:running_example_CUOM_theta_class}(a).
\end{rexample}

\begin{figure}[htb]
	\centering
	\includegraphics[width=.8\textwidth]{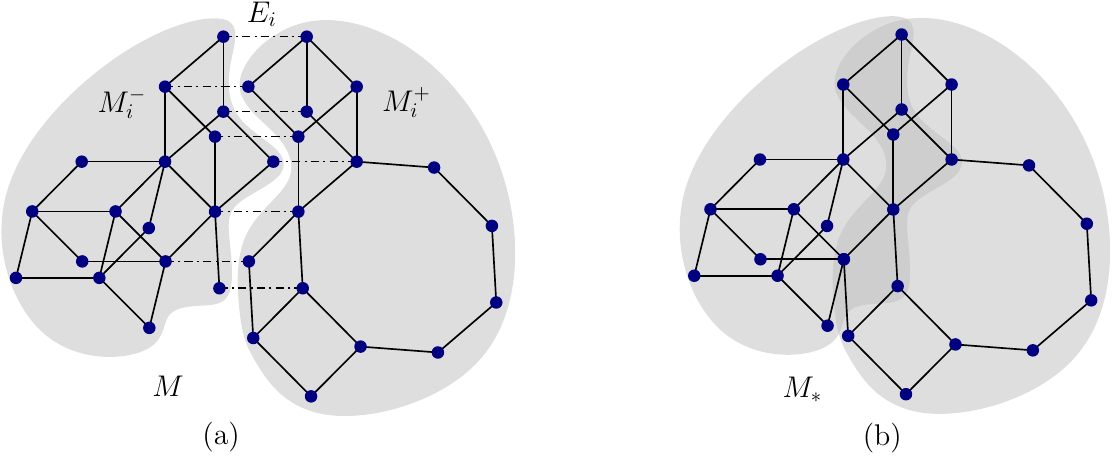}
	\caption{(a) Two complementary halfspaces of $M$. (b) The contraction $M_*$
	of a $\Theta$-class of $M$.} 
	\label{fig:running_example_CUOM_theta_class}
\end{figure}

\subsection{OMs, COMs, and AMPs}  We recall the basic notions and results from
the theory of oriented matroids (OMs), complexes of oriented matroids
(COMs), and ample set systems (AMPs). We follow \cite{BjLVStWhZi} for OMs,
\cite{BaChKn} for COMs, and \cite{BaChDrKo,La} for AMPs.

\subsubsection{OMs: oriented matroids}\label{sub:OM}
Oriented matroids (OMs) are abstractions of systems of sign vectors of all cells in the partition of ${\mathbb R}^m$ by a central arrangement of hyperplanes.
Let $U$ be a set with $m$ elements and let $\covectors$ be a {\it system of
sign vectors}, i.e., maps from $U$ to $\{-1,0,+1\}$. The elements of
$\covectors$ are referred to as \emph{covectors} and denoted by capital letters
$X, Y, Z$, etc. For $X \in \covectors$, the subset $\underline{X} = \{e\in U:
X_e\neq 0\}$ is called the \emph{support} of $X$ and  its complement
$X^0=E\setminus \underline{X}=\{e\in E: X_e=0\}$ the \emph{zero set} of $X$.
For a sign vector $X$ and a subset $A\subseteq U$, let $X_A$ be the restriction
of $X$ to $A$.  Let $\leq$ be the product ordering on $\{-1,0,+1\}^{U} $
relative to the standard ordering of signs with $0 \leq -1$ and $0 \leq +1$.
For $X,Y\in \covectors$, we call $S(X,Y)=\{e\in U: X_eY_e=-1\}$ the
\emph{separator} of $X$ and $Y$.
The \emph{composition} of $X$ and $Y$ is the sign vector $X\circ Y$, where for all $e\in U$, $(X\circ Y)_e = X_e$ if $X_e\ne 0$  and  $(X\circ Y)_e=Y_e$ if $X_e=0$.

\begin{definition} \label{def:OM}
	An \emph{oriented matroid} (OM) is a system of sign vectors
	$(U,\covectors)$ satisfying
	\begin{itemize}
		\item [{\bf (C)}] ({\sf Composition)} $X\circ Y \in \covectors$ for all
		$X,Y \in \covectors$.
		\item[{\bf (SE)}] ({\sf Strong elimination}) for each pair
		$X,Y\in\covectors$ and for each $e\in S(X,Y)$, there exists $Z \in
		\covectors$ such that $Z_e=0$ and $Z_f=(X\circ Y)_f$ for all $f\in
		U\setminus S(X,Y)$.
		\item[{\bf (Sym)}] ({\sf Symmetry}) $-\covectors=\{ -X: X\in
		\covectors\}=\covectors,$ that is, $\covectors$ is closed under sign
		reversal.
	\end{itemize}
\end{definition}

A system of sign-vectors $(U,\covectors)$ is \emph{simple} if it has no ``redundant'' elements, i.e., for each $e \in U$,
$\{X_e: X\in \covectors\}=\{-1,0,+1\}$ and for each pair $e\neq f$ 
there exist $X,Y \in \covectors$ with $\{X_eX_f,Y_eY_f\}=\{+, -\}$. We will only consider simple OMs, without explicitly stating it every time. The poset $(\covectors,\le)$ of an OM $\covectors$ together with an artificial global maximum $\hat{1}$ forms a graded lattice, called the \emph{big face lattice} $\mathcal{F}_{\mathrm{big}}(\covectors)$. The length of the maximal chains of $\mathcal{F}_{\mathrm{big}}$  minus one is called the \emph{rank} of $\covectors$ and denoted $\rk(\covectors)$. Note that $\rk(\covectors)$ equals the rank of the underlying unoriented matroid~\cite[Theorem 4.1.14]{BjLVStWhZi}. The \emph{topes} of $\covectors$ are the co-atoms of $\mathcal{F}_{\mathrm{big}}(\covectors)$.

From (C), (Sym), and (SE) it follows that the set ${\mathcal T}$ of topes of any simple OM $\covectors$ are  $\{-1,+1\}$-vectors. Thus,  ${\mathcal T}$ can be viewed as a family of subsets of $U$, where for each $T\in\mathcal{T}$ an element $e\in U$ belongs to the corresponding set if $T_e=+$ and does not belong to the set otherwise.
The \emph{tope graph} $G(\covectors)$ of an OM $\covectors$ is the 1-inclusion
graph of the set $\mathcal T$ of topes of $\covectors$. 
In \emph{realizable OMs} (i.e., of OMs arising from central hyperplane arrangements of ${\mathbb R}^m$), $X\le Y$ for two covectors $X,Y$  if and only if the cell corresponding to $X$ is contained in the cell corresponding to $Y$. Consequently, the
topes of realizable OMs are the covectors of the inclusion maximal cells
(which all have dimension $m$), called \emph{regions}. Therefore, the
tope graph of a realizable OM can be viewed as the adjacency graph of regions:
the vertices of
this graph are the regions of a hyperplane arrangement and two regions are
adjacent in this graph if they are separated by a unique hyperplane of the
arrangement. The \emph{Topological Representation Theorem of Oriented
Matroids} of
\cite{FoLa}, generalizes this correspondence to all OMs:
tope graphs of OMs can be characterized as the adjacency graphs of regions (inclusion maximal cells) of pseudo-sphere
arrangements in a sphere $S^{m-1}$~\cite{BjLVStWhZi}, where $m$ is the rank of
the OM. More precisely, two topes are adjacent if and only if the corresponding 
regions are separated by a unique pseudo-sphere, see Fig. 
\ref{fig:region_graph}.
It is also well-known (see for example~\cite{BjLVStWhZi}) that tope
graphs of OMs are partial cubes and that $\covectors$ can be recovered from its
tope graph  $G(\covectors)$ (up to isomorphism). Therefore, \emph{we can define
all terms in the language of tope graphs.} One instance of this is the correspondence between the rank of an OM and the
VC-dimension of its tope graph, see Lemma \ref{VCdimOM}. 
%

Another important axiomatization of OMs is in terms of {cocircuits}. The
\emph{cocircuits} of $\covectors$ are the minimal non-zero elements of
$\mathcal{F}_{\mathrm{big}}(\covectors)$, i.e., its atoms. The collection of
cocircuits is denoted by $\cocircuits$ and can be axiomatized
as follows:
a system of sign vectors  $(U,\cocircuits)$ is called an \emph{oriented
matroid} (OM) if  $\cocircuits$ satisfies {\bf (Sym)} and  the following two
axioms: 
\begin{itemize}
	\item [{\bf (Inc)}] ({\sf Incomparability)} $\underline{X}\subseteq\underline{Y}$
	implies $X=\pm Y$ for all $X,Y \in  \cocircuits$.
	\item[{\bf (E)}] ({\sf Elimination}) for each pair $X,Y\in\cocircuits$ with
	$X\neq -Y$ and for each $e\in S(X,Y)$, there exists $Z \in  \cocircuits$
	such that $Z_e=0$  and $Z_f\in\{0,X_f,Y_f\}$ for all $f\in U$.
\end{itemize}

The set $\covectors$ of covectors can be derived from $\cocircuits$ by taking
the closure of $\cocircuits$ under composition. The axiomatization of OMs via
cocircuits is used 
to define uniform oriented matroids.

\medskip
\begin{definition} ~\cite{BjLVStWhZi}
	A \emph{uniform oriented matroid} (UOM) of rank $r$ on a set $U$ of size
	$m$ is an OM $(U,\cocircuits)$ such that $\cocircuits$ consists of two
	opposite signings of each subset of $U$ of size $m-r+1$. 
\end{definition}

\subsubsection{COMs: complexes of oriented matroids}\label{subsec:COMs}
Complexes of oriented matroids (COMs) are abstractions of sign vectors of all cells in the partition of an open convex set $C$ of ${\mathbb R}^m$ by an arrangement of affine hyperplanes.
COMs are defined in a similar way as OMs, simply replacing the global axiom (Sym) by a weaker local axiom (FS) of face symmetry:

\begin{definition} \label{def:COM}
	A \emph{complex of oriented matroids} (COMs) is a system of sign vectors
	$(U,\covectors)$ satisfying (SE)  and the following axiom:
	\begin{itemize}
		\item[{\bf (FS)}] ({\sf Face symmetry}) $X\circ -Y \in  \covectors$
		for all $X,Y \in  \covectors$.
	\end{itemize}
\end{definition}

As for OMs, we restrict ourselves to \emph{simple} COMs, i.e., COMs defining simple systems of sign vectors. 
One can  see that (FS) implies (C), thus OMs are exactly the COMs containing the zero  vector ${\bf 0}$, see~\cite{BaChKn}.
A COM $\covectors$  is \emph{realizable} if $\covectors$ is the  system of sign vectors of all cells  in an arrangement $U$ of (oriented) affine hyperplanes
restricted to an open convex set of ${\mathbb R}^m$. 
For other examples of (tope graphs of) COMs, see~\cite{BaChKn,ChKnMa,KnMa}.

The simple twist between (Sym) and (FS) leads to a rich combinatorial and
geometric structure that is build from OMs but is quite different from OMs. Let
$(U,\covectors)$ be a COM and $X$ be a covector of $\covectors$. The
\emph{face} of $X$ is $\F(X):=\{ X\circ Y: Y\in \covectors\}$; a \emph{facet}
is a maximal proper face. From the definition, any face $\F(X)$  consists of
the sign vectors of all faces of the subcube of $[-1,+1]^{U}$ with barycenter
$X$. By \cite[Lemma 4]{BaChKn}, each face $\F(X)$ of $\covectors$ is an OM.
Since OMs are COMs, each face of an OM is an OM and the facets correspond to
cocircuits.
Furthermore, by \cite[Section 11]{BaChKn} replacing each combinatorial face
$\F(X)$ of  $\covectors$  by a PL-ball, we obtain a contractible cell complex
associated to each COM. The \emph{topes} and the \emph{tope graphs} of COMs are
defined in the same way as for OMs. Again, the topes are $\{-1,+1\}$-vectors,
the tope graph $G(\covectors)$ is a partial cube, and the COM $\covectors$ can
be recovered from its tope graph, see~\cite{BaChKn,KnMa}.

\begin{example}
	In Fig. \ref{fig:region_graph}(a) we consider an oriented pseudoline
	arrangement $U$. The set of sign vectors of the cells of this arrangement
restricted to the disk $B$ is denoted by $\covectors$. Notice that $\covectors$ is a COM. This is due
to the fact that $U$ comes from an arrangement $U'$ of pseudo-spheres of the sphere $S^2$,
which together with the boundary $\partial B$ of $B$ define an OM $\covectors'$. Then $\covectors$ is one of the  halfspaces
of $\covectors'$ defined by $\partial B$, thus $\covectors$ is a COM by \cite[Proposition 6]{BaChKn}.

In Fig. \ref{fig:region_graph}(a) we present the sign vectors $X$ and $Y$ of two cells, one included in another one.
Note that $X$ is a tope of $\covectors$ and $Y\leq X$. In Fig. 
\ref{fig:region_graph}(b), we present
the adjacency graph of the regions of the arrangement $U$ and draw in
red the covectors belonging to the face $F(Y)$ of $Y$.
In Fig. \ref{fig:region_graph}(c) the tope graph of the COM $\covectors$ is 
given. It is the adjacency graph of regions.

\end{example}

\begin{figure}[htb]
	\centering
	\includegraphics[width=.9\textwidth]{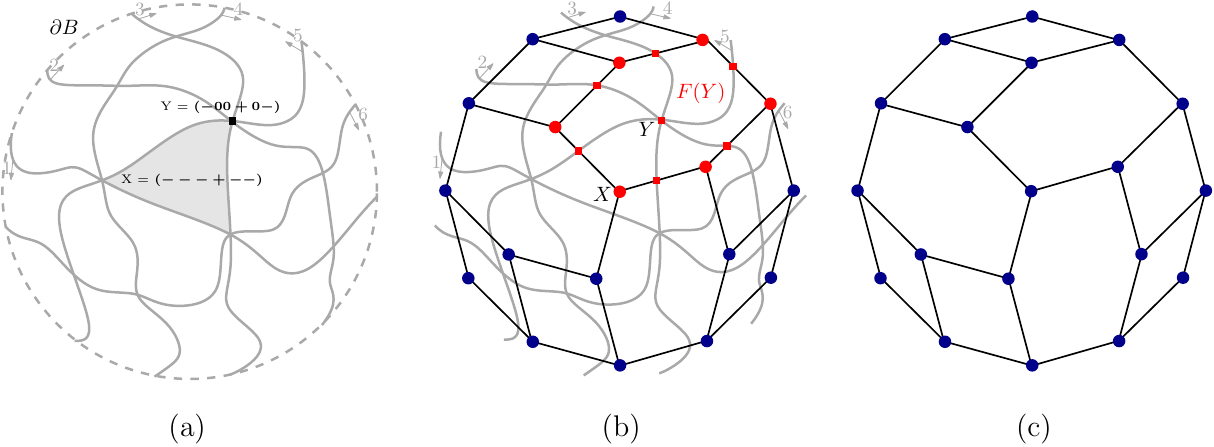}
	\caption{(a) A pseudoline arrangement $U$. (b) The adjacency graph of
	regions of $U$ and the face $F(Y)$. (c) The tope graph of COM $\covectors$.}
	\label{fig:region_graph}
\end{figure}



We continue with the definitions of AMPs and CUOMs.
For  $\covectors\subseteq \{-1,0,+1\}^U$,  let $\uparr \covectors:=\{ Y \in
\{-1,0,+1\}^{U}: X \leq Y \text{ for some } X \in  \covectors\}$.

\begin{definition} \label{def:AMP} ~\cite{BaChKn}
	An \emph{ample set system} (AMP) is a COM $(U,\covectors)$ satisfying the
	following axiom:
	\begin{itemize}
		\item[{\bf (IC)}] ({\sf Ideal composition})
		$\uparr\covectors=\covectors$.
	\end{itemize}
\end{definition}

\begin{definition} \label{def:CUOM}
	A \emph{complex of uniform oriented matroids} (CUOM) is a COM
	$(U,\covectors)$ in which each facet is a UOM.
\end{definition}

\subsubsection{AMPs: ample set systems}
Ample set systems (AMPs) are abstractions of sign vectors of all cells  in a partition of an open
convex set $C$ of ${\mathbb R}^m$ by an arrangement of coordinate hyperplanes. We
already defined ample set systems (1) as COMs $\covectors$ such
that $\uparr\covectors=\covectors$  and (2) as families of sets $\SS$ for which
the upper bounds in the sandwich lemma are tight: $\lvert \SS\rvert=\lvert
\oX(\SS)\rvert$. The set system arising in the second definition is  the set of
topes of the system of sign vectors from the first definition. As in the case of
OMs and COMs, \emph{we can consider AMPs as set systems or as partial cubes.}
By \cite{BaChDrKo,BoRa}, $\SS$ is ample if and only if
$\lvert\uX(\SS)\rvert=\lvert \SS\rvert$ and if and only if
$\uX(\SS)=\oX(\SS)$. This can be rephrased in the following combinatorial way:
$\SS\subseteq 2^U$ is \emph{ample} if and only if each set shattered by $\SS$
is strongly shattered.
Consequently, the VC-dimension of an ample set system is the dimension of the
largest cube in its 1-inclusion graph. A nice characterization of ample set
systems was provided in  \cite{La}: $\SS$ is ample if and only if for any cube
$\Q$ of $\Q_m$ if $\Q\cap \SS$ is closed under taking antipodes, then either
$\Q\cap \SS=\varnothing$ or $\Q$ is included in $G(\SS)$.  The paper
\cite{BaChDrKo} provides metric and recursive characterizations of ample
set systems, such as the following one. 

Recall that any two $X$-cubes $Q',Q''$ of $\SS$ are called \emph{parallel
cubes}. The distance $d(Q',Q'')$ is the distance between mutually closest
vertices of $Q'$ and $Q''$. A \emph{gallery of length $k$} between $Q'$ and
$Q''$ is a sequence of $X$-cubes $(Q'=R_0,R_1\ldots,R_{k-1},R_k=Q'')$ of $\SS$
such that $R_{i-1}\cup R_i$ is a cube for every $i=1,\ldots,k$. A
\emph{geodesic gallery} is a gallery of length $d(Q',Q'')$.

\begin{proposition} \label{ample-gallery} \cite{BaChDrKo}
	$\SS$ is ample if and only if any two parallel cubes of $\SS$ can be
	connected in $\SS$ by a geodesic gallery.
\end{proposition}

Thus, the 1-inclusion graph $G(\SS)$ of an ample set $\SS$ is a partial cube and
we will speak about \emph{ample partial cubes}. We conclude with the definition
of ample completions.

\begin{definition} \label{def:amp-completion}
	An \emph{ample completion} of a subgraph $G$ of VC-dimension $d$ of $\Q_m$
	is an ample partial cube $\ac(G)$ containing $G$ as a subgraph and such
	that $\vcd(\ac(G))=d$.
\end{definition}


\subsection{Sample compression schemes}
\label{subsect_sample_compression_schemes}
The language of sign vectors is perfectly suited for defining  
sample compression schemes. This reformulation is due to \cite{ChChMInRaVa}, which we closely 
follow (for classical formulations, see \cite{LiWa,MoWa,MoYe}). A \emph{concept class} is any family $\mathcal C$ of subsets 
of a finite set $U$; the elements of $\mathcal C$ are called \emph{concepts}. As we noticed above, $\mathcal C$ can be viewed as a subset 
of $\{ -1,+1\}^U$. For a concept $C\in {\mathcal C}$, any covector $X\in \{ -1,0,+1\}^U$ such that  $X\le C$ is called a 
\emph{sample realizable by} $C$. Let $\downarrow {\mathcal C}$ denote the set of all samples realizable by concepts of $\mathcal C$, i.e., 
$\downarrow {\mathcal C}=\{ X\in \{ -1,0,+1\}^U: \exists C\in {\mathcal C} \mbox{ such that } X\le C\}$. 

A \emph{labeled sample compression scheme} for a concept class $\mathcal C$ is 
best
viewed as a protocol between a \emph{compressor} and a
\emph{reconstructor}.  The compressor gets a realizable sample $X$
from which it picks a small \emph{subsample} $X' \le X$.
The compressor sends $X'$ to the reconstructor.  Based on $X'$, the reconstructor outputs a
set $C\in \{ -1,+1\}^U$ (not necessarily belonging to $\mathcal C$) that needs 
to be \emph{consistent} with the entire 
sample $X$, i.e., $X \le C$.  An \emph{unlabeled sample compression scheme} is a
sample compression scheme in which the compressed subsample $X'$ is
unlabeled. So, the compressor removes the labels before sending the
subsample to the reconstructor. In both the labeled and unlabeled settings, the 
goal is to minimize the size of 
the support of the compressed subsample $X'$ with $X$ running over $\downarrow {\mathcal C}$.  

Formally, a \emph{labeled sample compression scheme}  of size $k$ for a concept class
${\mathcal C} \subseteq \{-1,+1\}^U$ is defined by a compressor function
$\alpha: \{ -1,0,+1\}^U \to \{ -1,0,+1\}^U$ and
a reconstructor function
$\beta: \{ -1,0,+1\}^U  \to \{ -1, +1\}^U$ such that for any
$X\in \downarrow\mathcal C$, it holds:
$$\alpha(X)\le X\le \beta(\alpha(X)) \mbox{ and } |\underline{\alpha}(X)|\le k,$$
where $\le$ is the order between sign vectors and $\underline{\alpha}(X)$ is 
the support of the sign vector $\alpha(X)$ as defined above. The condition 
$X\le \beta(\alpha(X))$ means that 
the restriction of $\beta(\alpha(X))$ on the support of $X$ coincides with the input sample $X$. In particular, if $X\in\mathcal{C}$, then $\beta(\alpha(X))=X$.
The \emph{unlabeled sample compression schemes} are defined analogously, with the difference 
that in the unlabeled case $\alpha(X)$ is a subset of size at most $k$ of the support of $X$. 

From the definition it follows that if ${\mathcal C}$ is a completion of a concept class 
${\mathcal C}'$ and $(\alpha,\beta)$ is a labeled, respectively unlabeled, 
sample compression scheme for 
$\mathcal C$, then $(\alpha,\beta)$ is a labeled, respectively unlabeled, 
sample compression scheme for ${\mathcal C}'$. 
(This conclusion is no longer true if the reconstructor map $\beta$ takes values in 
$\mathcal C$ and not in $\{ -1,+1\}^U$.) 

The sample compression conjecture of \cite{FlWa} states that {\it any set 
family 
of VC-dimension $d$ has a sample compression scheme of size $O(d)$}.

\section{Auxiliary results}
In this section we recall or prove some auxiliary results used in the proofs of the main results.

\subsection{Partial cubes and VC-dimension}
In this subsection, we closely follow \cite{ChKnMa} and \cite{ChKnPh}. 
Let $G$ be a partial cube, isometrically embedded in the hypercube $\Q_m$.

\subsubsection{pc-Minors and VC-dimension}
For a $\Theta$-class $E_i$ of a partial cube $G$,  an {\it elementary
restriction} consists of taking one of the halfspaces $G^-_i$ and $G^+_i$. More
generally, a {\it restriction} is a convex subgraph of $G$ induced by the
intersection of a set of halfspaces of $G$. 
Since any convex subgraph of a partial cube $G$ is the intersection of
halfspaces \cite{AlKn,Ch_thesis}, the restrictions of $G$ coincide with the
convex subgraphs of $G$. For a $\Theta$-class $E_i$, the graph  $\pi_i(G)$
obtained from $G$ by contracting the edges of $E_i$ is called an ($i$-){\it
contraction} of $G$. For a vertex $v$ of $G$, let $\pi_i(v)$ be the image of
$v$ under the $i$-contraction. We will apply $\pi_i$ to subsets $S\subseteq V$,
by setting $\pi_i(S):=\{\pi_i(v): v\in S\}$.
In particular, we denote the $i$-{\it contraction} of $G$ by $\pi_i(G)$.

\begin{rexample} The $i$-contraction $M_*=\pi_i(M)$  of the running example $M$ is
given in Fig. \ref{fig:running_example_CUOM_theta_class}(b); $M_*$ is a CUOM and thus is a
partial cube.
\end{rexample}

By~\cite[Theorem 3]{Ch_hamming}, $\pi_i(G)$ is an isometric subgraph of
$\Q_{m-1}$, thus the class of partial cubes is closed under contractions. Since
edge contractions in graphs commute, if $E_i,E_j$ are two distinct
$\Theta$-classes, then $\pi_j(\pi_i(G))=\pi_i(\pi_j(G))$. Consequently, for a
set $A$ of $\Theta$-classes, we can denote by $\pi_A(G)$ the isometric subgraph
of $\Q_{m-|A|}$ obtained from $G$ by contracting the equivalence classes of
edges from $A$. Contractions and restrictions commute in partial cubes: any set
of restrictions and contractions of a partial cube $G$ provide the same result,
independently on the order they are performed in. The resulting partial cube is
called a {\it partial cube minor} (or {\it pc-minor}) of $G$. For a partial
cube $H$, let ${\mathcal F}(H)$ denote  the class of all partial cubes
not having $H$ as a pc-minor. For partial cubes, 
the VC-dimension can be formulated in terms of pc-minors:

\begin{lemma} \label{VCdim_d} \cite[Lemma 2]{ChKnPh}
	A partial cube $G$ has VC-dimension $\le d$ if and only if $G\in {\mathcal
	F}(\Q_{d+1})$.
\end{lemma}

\begin{rexample}
The rhombododecahedron $D$ from Fig. \ref{fig:RDinQ4} is a tope graph of a UOM
and is a
facet of the running example $M$. In Fig. \ref{fig:RDinQ4}, $D$ is isometrically
embedded in the 4-cube $Q_4$. Since $D$ is a proper subgraph of $Q_4$, $\vcd(D)<4$. On the other hand, contracting any $\Theta$-class of
$D$, we obtain the 3-cube $Q_3$. In Fig. \ref{fig:RDinQ4}, we represent the
contraction of the $\Theta$-class $E_i$ constituted by oblique edges. The
vertices that are merged by this contraction are darker.
Thus $Q_3$ is a pc-minor of $D$, establishing that $\vcd(D)=3$.
The 8-cycle $C_8$ is another face of $M$. Analogously to $D$ one can show that $\vcd(C_8)=2$: the contraction of any two of its $\Theta$-classes
results in the square $Q_2$ and the contraction of any single $\Theta$-class does not yield $Q_3$.
\end{rexample}

\begin{figure}[htb]
	\centering
	\includegraphics[width=.55\textwidth]{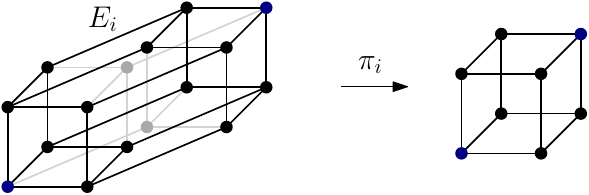}
	\caption{The rhombododecahedron $D$ isometrically embedded in $Q_4$ and
	its resulting pc-minor $\pi_i(D)$ after contracting the $\Theta$-class
	$E_i$ represented by the oblique edges.}
	\label{fig:RDinQ4}
\end{figure}

An \emph{antipode} of a vertex $v$ in a partial cube $G$ is the (necessarily unique) vertex $-v$ such that $G=I(v,-v)$.
A partial cube $G$ is \emph{antipodal} if all its vertices have antipodes. For
a subgraph $H$ of an antipodal partial cube $G$ we  denote by $-H$ the set of
antipodes of $H$ in $G$. We will use several times the following two results:

\begin{lemma}\label{lem:antipodal} \cite[Lemma 16]{ChKnPh}
	If $G$ is a  proper convex subgraph of an antipodal partial cube
	$H\in\mathcal{F}(\Q_{d+1})$, then $G\in\mathcal{F}(\Q_{d})$.
\end{lemma}

\begin{lemma}\label{lem:antipodal_clos_pi}\cite{KnMa}
	Antipodal partial cubes are closed under contractions.
\end{lemma}

The following lemma (a direct consequence of Theorem \ref{Djokovic})
is used in Proposition \ref{AMPamalgam}:

\begin{lemma} \label{convex-intervals}
	Intervals of partial cubes are convex.
\end{lemma}

\subsubsection{Shattering via Cartesian products}
Recall that the \emph{Cartesian product} of $m$ graphs
$G_1,\ldots,G_m$ is the graph
$G=G_1\product\ldots\product G_m$ whose vertex-set consists of all $m$-tuples $(v_1,\ldots,v_m)$ with $v_i\in V(G_i)$ and two
$m$-tuples $u=(u_1,\ldots,u_m)$ and $v=(v_1,\ldots,v_m)$ are adjacent in $G$ if
and only if there exists an index $1\le i\le m$ such that $u_i$ is adjacent to
$v_i$ in $G_i$ and
$u_j=v_j$ for all $j\ne i$.
The $m$-cube $\Q_m=\Q(U)$ is the Cartesian product of $m$ copies of $K_2$,
i.e., $\Q_m=K_2\product \cdots \product K_2$. 
For a subset $X\subseteq U$ of size $d$, denote by $\Q(X)$ the Cartesian
product of the factors of $\Q(U)$ indexed by the elements of $X$; clearly,
$\Q(X)$ is a $d$-cube $\Q_d$. Analogously, let $\Q(U\setminus X)$ be the
$(m-d)$-cube defined by $U\setminus X$. Then $\Q(U)=\Q(X)\product \Q(U\setminus
X)$. For $Y\subseteq U$, let $\Q_Y$ be the subgraph of $\Q(U)$ induced by the
sets $\{ Y\cup Z: Z\subseteq U\setminus Y\}$; each $\Q_Y$ is isomorphic to the
cube $\Q(U\setminus X)=\Q_{m-d}$. Let $G$ be an isometric subgraph of $\Q(U)$.
We also denote by $G_Y$ the intersection of $G$ with the cube $\Q(Y)$ and we
call $G_Y$ the \emph{$Y$-fiber} of $G$. Since each $\Q_Y$ is a convex subgraph
of $\Q(U)$ and $G$ is an isometric subgraph of $\Q(U)$, each fiber $F_Y$ is
either empty or a non-empty convex subgraph of $G$. Then the definition of
shattering can be rephrased in the following way:

\begin{lemma} \label{shattering-fibers}
	A subset $X$ of $U$ is shattered by an isometric subgraph $G$ of $\Q(U)$ if
	and only if  each fiber $G_Y$, $Y\subseteq X$, is a nonempty convex
	subgraph of $G$.
\end{lemma}

If $X\subseteq U$ is shattered by $G$, we call the map $\psi: V(G)\rightarrow
2^X=V(\Q(X))$ such that $\psi^{-1}(Y)=G_Y$ for all $Y\subseteq X$, a
\emph{shattering map}. The edges of $G$ between two different fibers are called
\emph{$X$-edges}. Note that for any $X$-edge $uv$ of $G$ there exists $e\in X$
and $Y\subset X$ such that $u$ corresponds to $Y$ and $v$ corresponds to $X\cup
\{ e\}$. Since the fibers define a partition of the vertex-set of $G$,  any
path  connecting two vertices from different fibers of $G$   contains $X$-edges.

The following simple lemmas are well-known and will be used in the proof of Theorem \ref{CUOMtoAMP}:

\begin{lemma} \label{path-theta}
	A $(u,v)$-path $P$ of a partial cube $G$ is a shortest path if and only if
	all edges of $P$ belong to different  $\Theta$-classes of $G$.
\end{lemma}

\begin{proof}
	Suppose $|P\cap E_i|\ge 2$ and  let $xy$ and $y'x'$ be edges of $P\cap E_i$
	that are consecutive with respect to $P$. Then $x,x'$  belong to the same
	halfspace $H^-_i$ or $H^+_i$ and $y,y'$ belong to the complementary
	halfspace. Since $y,y'\in I(x,x')$, this contradicts Theorem \ref{Djokovic}.
\end{proof}

\begin{lemma} \label{gated-theta}
	If $G$ is a partial cube, $H$ a gated subgraph of $G$, $v$ a vertex of $G$,
	and $v'$ the gate of $v$ in $H$, then no shortest $(v,v')$-path of $G$
	contains edges of a $\Theta$-class of $H$.
\end{lemma}

\begin{proof}
	Suppose a shortest $(v,v')$-path $P$ contains an edge $zz'$ of a
	$\Theta$-class $E_i$ of $H$. Let $xy$ be an edge of $H$ belonging to $E_i$.
	Since $G$ is bipartite, let $d(v',x)<d(v',y)$. Since $v'$ is the gate of
	$v$ in $H$, the path $R$ constituted by $P$, a shortest $(v',x)$-path of
	$H$, and the edge $xy$ is a shortest $(v,y)$-path of $G$. Since $R$
	contains two edges of $E_i$, $R$ cannot be a shortest path.
\end{proof}

\subsubsection{Isometric expansions and VC-dimension}
A triplet $(G^1,G^0,G^2)$ is called an {\it isometric cover} of a connected graph $G$, if the following conditions are satisfied:
\begin{itemize}
	\item $G^1$ and $G^2$ are two isometric subgraphs of $G$;
	\item $V(G)=V(G^1)\cup V(G^2)$ and $E(G)=E(G^1)\cup E(G^2)$;
	\item $V(G^1)\cap V(G^2)\ne \varnothing$ and $G^0$ is the
	subgraph of $G$ induced by  $V(G^1)\cap V(G^2)$.
\end{itemize}
A graph $G'$ is an {\it isometric expansion} of  $G$ with respect to an
isometric cover $(G^1,G^0,G^2)$ of $G$ (notation $G'=\psi(G)$) if $G'$ is
obtained from $G$ by replacing each vertex $v$ of $G^1$ by a vertex $v_1$ and
each vertex $v$ of $G^2$ by a vertex $v_2$ such that $u_i$ and $v_i$, $i=1,2$
are adjacent in $G'$ if and only if $u$ and $v$ are adjacent vertices of $G^i$
and $v_1v_2$ is an edge of $G'$ if and only if $v$ is a vertex of $G^0$.
If  $G^1=G^0$ (and thus $G^2=G$), then the isometric expansion is called
\emph{peripheral} and we say that $G'$ is obtained from $G$ by a peripheral
expansion with respect to $G^0$.

\begin{rexample} Recall that $M_*$ is a pc-minor of the graph $M$  and is obtained by contracting a single $\Theta$-class of $M$. In
Fig. \ref{fig:running_example_CUOM_expansion_isometric} we present an isometric expansion  of $M_*$, resulting in a partial cube $M'_*$ different from $M$.
The initial graph $M$ can be retrieved from $M_*$ by an isometric expansion  with respect to the isometric cover of $M_*$ given in Fig. \ref{fig:running_example_CUOM_theta_class}(b).
\begin{figure}[htb]
	\centering
	\includegraphics[width=.8\textwidth]{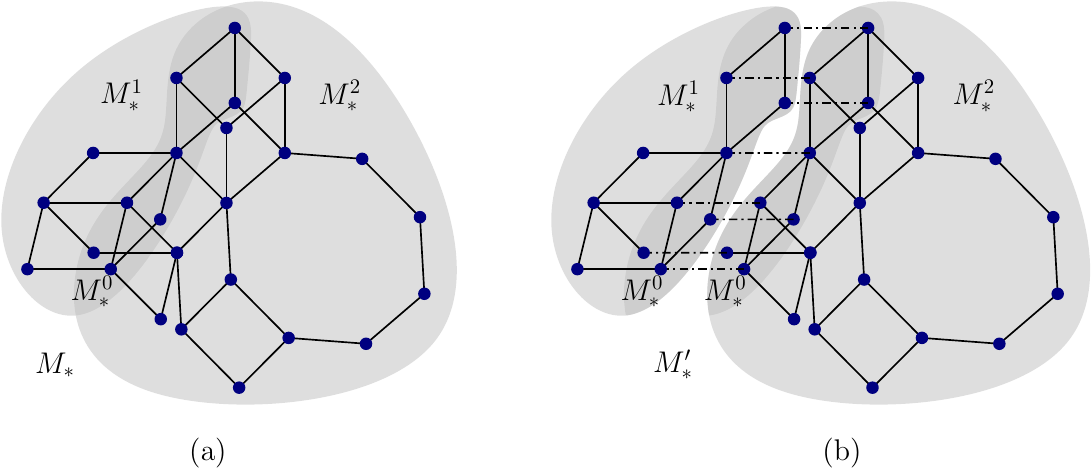}
	\caption{(a) The pc-minor $M_*$ of $M$. (b) An isometric expansion $M'_*$ of $M_*$.}
	\label{fig:running_example_CUOM_expansion_isometric}
\end{figure}
\end{rexample}

If $(G^1,G^0,G^2)$ is not peripheral, then $G^0$ is a separator of
$G$. By~\cite{Ch_thesis,Ch_hamming}, $G$ is a partial cube if and only if $G$
can be obtained by a sequence of isometric expansions from a single vertex.

%

There is an intimate relation between contractions and isometric expansions. If
$G$ is a partial cube and $E_i$ is a $\Theta$-class of $G$, then contracting
$E_i$ we obtain the pc-minor $\pi_i(G)$ of $G$. Then $G$ can be obtained from
$\pi_i(G)$ by an isometric expansion with respect to $(\pi_i(G^+_i),
G_0,\pi_i(G^-_i))$, where $\pi_i(G^+_i)$ and  $\pi_i(G^-_i)$ are the images by
the contraction of the halfspaces $G^+_i$ and $G^-_i$ of $G$ and $G_0$ is the
contraction of the vertices of $G$ incident to edges from $E_i$. The following
result is used in the proof of Theorem \ref{OMtoAMP}:

\begin{proposition} \label{expansion-Qd+1} \cite[Proposition 15]{ChKnPh}
	Let $G'$ be obtained from $G\in {\mathcal F}(\Q_{d+1})$ by an isometric
	expansion with respect to $(G^1,G^0,G^2)$. Then $G'\in {\mathcal
	F}(\Q_{d+1})$ if and only if $\vcd(G^0)\le d-1$.
\end{proposition}

\subsection{OMs, COMs, and AMPs}
Here we recall some results about OMs, COMs, and AMPs.

\subsubsection{Faces}
First, since OMs satisfy the axiom (Sym), we obtain:

\begin{lemma} \label{OM-antipodal}
	The tope graph of any OM  is an antipodal partial cube.
\end{lemma}

Let $(U,\covectors)$ be a COM. For a covector $X\in \covectors$, recall that
$\F(X)$ 
denotes the \emph{face} of $X$. Let also $\C(X):=\C(F(X))$ denote the smallest cube  of $\Q(U):=\{-1,+1\}^{U}$ containing $\F(X)$. 
Note that $\F(X)$ and $\C(X)$ are defined by the same set of $\Theta$-classes.
We continue with a fundamental property of faces of COMs:

\begin{lemma} \label{face-OM} \cite{BaChKn}
	For each covector $X$ of a COM $\covectors$, the face $\F(X)$ is an OM.
\end{lemma}

The following lemma is implicit in \cite{BaChKn}, explicit in \cite{KnMa}, and used in the proof of Theorem \ref{CUOMtoAMP}:

\begin{lemma} \label{face-gated}
	For each covector $X$ of a COM $\covectors$, the face $\F(X)$ defines a gated
	subgraph of the  tope graph $G$ of $\covectors$. Moreover,
	for any tope $Y$ of $\covectors$, $X\circ Y$ is  the gate of $Y$ in the
	cube $\C(X)$.
\end{lemma}

\begin{proof}
	Pick any $Y\in \{-1,+1\}^U\cap \covectors$. By the definition of $X\circ
	Y$,  $X\circ Y\in \{-1,+1\}^U$, thus $X\circ Y$ is a tope of $\covectors$.
	By definition of $\F(X)$, $X\circ Y$ belongs to $\F(X)$ (and thus to
	$\C(X)$). Since $(X\circ Y)_e=Y_e$ for all $e\in U\setminus \underline{X},$
	necessarily $X\circ Y$ is the gate of $Y$ in $\C(X).$
\end{proof}

\subsubsection{Minors and pc-minors} In  the present subsection we give two lemmas that are essential
for inductive proofs.
We start with the following result about pc-minors of tope graphs of COMs and
AMPs, which follows from the results of \cite{BaChDrKo} and \cite{BaChKn}:

\begin{lemma} \label{contractionsCOM-AMP}
	The classes of tope graphs of COMs and AMPs are closed under taking pc-minors.
	The class of tope graphs of  OMs is closed under contractions.
\end{lemma}

We continue with the notions of restriction,  contraction, and minors for COMs
(which can be compared  with the similar notions for partial cubes). Let
$(U,\covectors)$ be a COM and $A\subseteq U$. Given a sign vector
$X\in\{-1,0,+1\}^U$ by $X\setminus A$ we refer to the \emph{restriction} of $X$
to $U\setminus A$, that is $X\setminus A\in\{-1,0,+1\}^{U\setminus A}$ with
$(X\setminus A)_e=X_e$ for all $e\in U\setminus A$. The \emph{deletion} of $A$
is defined as $(U\setminus A,\covectors\setminus A)$, where
$\covectors\setminus A:=\{X\setminus A:  X\in\covectors\}$. The
\emph{contraction} of $A$ is defined as $(U\setminus A,\covectors/ A)$, where
$\covectors/ A:=\{X\setminus A: X\in\covectors\text{ and }\underline{X}\cap
A=\varnothing\}$. If $\covectors'$ arises by deletions and contractions from
$\covectors$, $\covectors'$ is said to be \emph{minor} of  $\covectors$.
Deletion in a COM translates to pc-contraction in its tope graph,
while contraction corresponds to what is called the zone graph,
see~\cite{KnMa}.

\begin{lemma}\label{lem:minorclosed}  \cite[Lemma 1]{BaChKn}
	The classes of COMs and AMPs are closed under taking  minors.
\end{lemma}

\subsubsection{Hyperplanes, carriers, and half-carriers}
For a COM $(U,\covectors)$, a {\it hyperplane} of $\covectors$ is the set
$\covectors_e^0:=\{ X\in \covectors: X_e=0\} \mbox{ for some } e\in U.$ The
{\it carrier} $N(\covectors_e^0)$ of the hyperplane $\covectors_e^0$ is the
union of all faces $\F(X')$ of $\covectors$ with $X'\in \covectors_e^0$. The
{\it positive and negative (open) halfspaces} supported by the hyperplane
$\covectors_e^0$ are $\covectors_e^+:= \{ X\in \covectors: X_e=+1\}$ and
$\covectors_e^-:=\{ X\in \covectors: X_e=-1\}.$ The carrier $N(\covectors_e^0)$
minus $\covectors_e^0$ splits into its positive and negative parts:
$N^+(\covectors_e^0):=\covectors^+_e\cap N(\covectors_e^0)$ and
$N^-(\covectors_e^0):=\covectors_e^-\cap N(\covectors_e^0),$ which we call
\emph{half-carriers}.

\begin{proposition} \label{carriers} \cite[Proposition 6]{BaChKn}
	In COMs and AMPs, all halfspaces, hyperplanes, carriers, and half-carriers
	are COMs and AMPs. In OMs, all hyperplanes and carriers are OMs.
\end{proposition}

The result about AMPs was not stated in  \cite[Proposition 6]{BaChKn}, however it easily follows from the definition of AMPs as  COMs satisfying the axiom (IC). The result also
follows from  \cite{BaChDrKo}. Proposition \ref{carriers} is used in the proof of the characterization of CUOMs from Proposition \ref{CUOM}.

\subsubsection{Amalgams}
One  important property of COMs is that they can be obtained by amalgams
from their maximal faces, i.e., they are amalgams of OMs.
Now  we make this definition precise. Following \cite{BaChKn}, we say that a
system $(U,\covectors)$ of sign vectors is a {\it COM-amalgam} of two COMs
$(U,\covectors')$ and $(U,\covectors'')$ if the following conditions are
satisfied:
\begin{itemize}
	\item[(1)] $\covectors=\covectors'\cup \covectors''$ with
	$\covectors'\setminus \covectors'', \covectors''\setminus \covectors',
	\covectors'\cap \covectors''\ne \varnothing$;
	
	\item[(2)] $(U,\covectors'\cap \covectors'')$ is a COM;
	
	\item[(3)] $\covectors'\circ \covectors''\subseteq \covectors'$ and
	$\covectors''\circ \covectors'\subseteq \covectors''$;
	
	\item[(4)] for $X\in \covectors'\setminus \covectors''$ and $Y\in
	\covectors''\setminus \covectors'$ with $X^0=Y^0$ there
	exists a shortest path in the tope graph $G(\covectors\setminus X^0)$ of
	the deletion $\covectors\setminus X^0$ of $X^0$.
\end{itemize}

We continue with two propositions: Proposition \ref{amalgamCOM} is used in the proof of Theorem \ref{CUOMtoAMP} and Proposition \ref{AMPamalgam}
is used in the proofs of both main results.

\begin{proposition} \label{amalgamCOM} \cite[Proposition 7 \& Corollary
2]{BaChKn}
	The COM-amalgam of  two COMs $\covectors',\covectors''$ is a COM
	$\covectors$ in which every facet is a facet of $\covectors'$ or
	$\covectors''$. Any COM which is not an OM is  obtained via successive COM-amalgams from its
	facets.
\end{proposition}


We will not make use of the following and state it here without proof just for completeness:
\begin{corollary} \label{amalgamAMP}
	The COM-amalgam of two AMPs $\covectors',\covectors''$ such that
	$\covectors'\cap \covectors''$ is ample is an AMP. Any AMP which is not a cube is obtained via
	successive AMP-amalgams from its maximal faces.
\end{corollary}


Now, we present a notion of AMP-amalgam formulated in terms of  graphs. We say
that a graph $G$ is an \emph{AMP-amalgam} of $G_1$ and $G_2$ if  $(G_1,G_1\cap
G_2,G_2)$ is an isometric cover of $G$ and $G_1,G_2,$ and $G_0=G_1\cap G_2\ne
G_1,G_2$ are ample partial cubes. The main difference between this and
COM-amalgams is that condition (4) in the definition of a COM-amalgam is
replaced by the weaker condition that $G$ is a partial cube. The next result
from \cite{BaChDrKounp} was never published:

\begin{proposition} \label{AMPamalgam} \cite{BaChDrKounp}
	Let $G$ be a subgraph of the hypercube $Q_m$ which is an AMP-amalgam of two
	ample isometric subgraphs $G_1$ and $G_2$ of $Q_m$. If $G$ is an isometric
	subgraph of $Q_m$, then $G$ is ample. Any ample partial cube can be
	obtained by AMP-amalgams from its facets.
\end{proposition}

\begin{proof}
	First we assert that any $X$-cube $Q$ of $G$ is contained either in $G_1$
	or in $G_2$. We proceed by induction on $k:=|X|$. Since $G_0=G_1\cap G_2$
	is a separator, the assertion holds when $k=1$. 
	Suppose the assertion is true for all $X'\subset U$ with $|X'|<k$ and
	suppose that the $X$-cube $Q$ of $G$ contains two vertices $s\in
	V(G_1)\setminus V(G_2)$ and $t\in V(G_2)\setminus V(G_1)$. By induction
	hypothesis,  any facet of $Q$ containing $s$ must be included in $G_1$ and
	any facet of $Q$ containing $t$ must be included in $G_2$. From this we
	conclude that all vertices of $Q$ except $s$ and $t$ (which must be
	opposite in $Q$) belong to $G_0$. This is impossible since $G_0$ is ample.
	
	By Proposition \ref{ample-gallery}, we must show that any two  $X$-cubes
	$Q_1,Q_2$ of  $G$ can be connected by a geodesic gallery.   Since $G_1$ and
	$G_2$ are ample, this is  true when $Q_1$ and $Q_2$ both belong  to $G_1$
	or to $G_2$. By previous assertion we can suppose that $Q_1\subseteq G_1$
	and $Q_2\subseteq G_2$.  By induction on $k=|X|$ we prove that $Q_1$ and
	$Q_2$ can be connected by a geodesic gallery containing an $X$-cube of
	$G_0$. If $k=0$, then $Q_1,Q_2$ are vertices of $G$ separated by $G_0$ and
	we are done. 
	So, let $k>0$. Pick any $e\in X$, set $X':=X\setminus \{ e\}$,  and let
	$G^+,G^-$ be the halfspaces of $G$ defined by $e$. Let $Q^+_1,Q^-_1$ and
	$Q^+_2,Q^-_2$ be the intersections of $Q_1$ and $Q_2$ with the  halfspaces.
	By induction hypothesis, $Q^+_1,Q^+_2$ can be connected by a geodesic
	gallery $P(Q^+_1,Q^+_2)$ containing an $X'$-cube $R^+$ in $G_0$ and
	$Q^-_1,Q^-_2$ can be connected by a geodesic gallery $P(Q^-_1,Q^-_2)$
	containing an $X'$-cube $R^-$ in $G_0$. Hence
	$d(Q^+_1,Q^+_2)=d(Q^+_1,R^+)+d(R^+,Q^+_2)$ and
	$d(Q^-_1,Q^-_2)=d(Q^-_1,R^-)+d(R^-,Q^-_2)$. Since $G^+$ and $G^-$ are
	convex subgraphs of $G$, $P(Q^+_1,Q^+_2)\subseteq G^+$ and
	$P(Q^-_1,Q^-_2)\subseteq G^-$. Since $G_0$ is ample, the $X'$-cubes $R^+$
	and $R^-$ can be connected in $G_0$ by a geodesic gallery.  Since
	$R^+\subseteq G^+$ and $R^-\subseteq G^-$, on this gallery we can find two
	consecutive $X'$-cubes $Q^+\subseteq G^+$ and $Q^-\subseteq G^-$ so that
	$Q=Q^+\cup Q^-$ is an $X$-cube of $G_0$.
	
	Since $Q_1$ and $Q$ are two $X$-cubes of AMP $G_1$, they can be connected
	in $G_1$ by a geodesic gallery $P(Q_1,Q)$. Analogously, $Q$ and $Q_2$ can
	be connected in $G_2$ by a geodesic gallery $P(Q,Q_2)$. We assert that the
	concatenation of the two galleries is a geodesic gallery $P(Q_1,Q_2)$
	between $Q_1$ and $Q_2$, i.e.,  $d(Q_1,Q_2)=d(Q_1,Q)+d(Q,Q_2)$.
	Since $d(Q_1^+,Q_2^+)=d(Q_1^-,Q_2^-)=d(Q_1,Q_2)$, it suffices to show that
	$d(Q_1^+,Q_2^+)=d(Q_1^+,Q^+)+d(Q^+,Q_2^+)$ and
	$d(Q_1^-,Q_2^-)=d(Q_1^-,Q^-)+d(Q^-,Q_2^-)$.
	
	In each $Q_1^+,Q_1^-,Q^+_2,Q^-_2,R^+,R^-$ pick a vertex, say $q_1^+\in
	Q_1^+,q_1^-\in Q_1^-,q_2^+\in Q^+_2, q_2^-\in Q^-_2,r^+\in R^+,r^-\in R^-$,
	such that each pair of vertices realizes the distance between the
	corresponding cubes. Then $d(q_1^+,q_1^-)=d(q_2^+,q_2^-)=1$ and
	$d(q_1^+,q_2^+)=d(q_1^+,r^+)+d(r^+,q_2^+)$ and
	$d(q_1^-,q_2^-)=d(q_1^-,r^-)+d(r^-,q_2^-)$.
	Let $q^+$ and $q^-$ be two vertices of $Q^+$ and $Q^-$, respectively,
	belonging to a shortest $(r^+,r^-)$-path. Again, $d(q^+,q^-)=1$.
	Consequently, in $G$ we have $r^+,r^-\in I(q_1^+,q_2^-)$ and $q^+,q^-\in
	I(r^+,r^-)$. Since $G$ is a partial cube, the interval $I(q_1^+,q_2^-)$ is
	convex (Lemma \ref{convex-intervals}), thus $q^+$ and $q^-$ belong to a
	common shortest path between $q_1^+$ and $q_2^-$. Applying the same
	argument, we deduce that $q^-$ and $q^+$ belong to a common shortest path
	between $q_1^-$ and $q_2^+$. Hence
	$d(q_1^+,q_2^+)=d(q_1^+,q^+)+d(q^+,q_2^+)$ and
	$d(q_1^-,q_2^-)=d(q_1^-,q^-)+d(q^-,q_2^-)$, establishing that
	$d(Q^+_1,Q^+_2)=d(Q_1^+,Q^+)+d(Q^+,Q_2^+)$ and
	$d(Q^-_1,Q^-_2)=d(Q_1^-,Q^-)+d(Q^-,Q_2^-)$. Consequently,
	$d(Q_1,Q_2)=d(Q_1,Q)+d(Q,Q_2)$, i.e., $P(Q_1,Q_2)$ is a geodesic gallery.
\end{proof}

\subsubsection{VC-dimension}

The VC-dimension of OMs, COMs, and AMPs (all viewed as partial cubes) can be expressed in the following way. Subsequently, we will use this fundamental
lemma  without explicitly mentioning it.

\begin{lemma} \label{VCdimOM}
	If $G$ is the tope graph of an OM $\covectors$, then $\vcd(G)=\rk(\covectors)$. If $G$ is the tope graph of a
	COM $\covectors$, then $\vcd(G)$ is the largest VC-dimension among tope graphs of faces of $\covectors$. If $G$ is the tope graph of an AMP $\covectors$, then $\vcd(G)$ is the largest dimension among cubes of
	$G$.
\end{lemma}

\begin{proof}
	Let $G$ be the tope graph of an OM $\covectors$. We have to prove that
	$\vcd(G)=\rk(\covectors)$. If $G=\Q_m$, then $\covectors=\{-1,0,+1\}^m$ has rank $m$
	and the equality holds. Thus, let $G$ be not a cube.  First we show
	$\vcd(G)\leq\rk(\covectors)$. Since $G$ is not a cube, it contains a $\Theta$-class
	$E_i$ whose contraction does not decrease the VC-dimension. If we set
	$G':=\pi_i(G)$ and  $\covectors':=\covectors\backslash \{ i\}$, then $\vcd(G')=\vcd(G)$ and $G'$ is the tope graph of $\covectors'$
(Lemma \ref{lem:minorclosed}). 
	Since $\rk(\covectors')\leq\rk(\covectors)$, by induction hypothesis, $\vcd(G)=\vcd(G')\leq
	\rk(\covectors')\leq\rk(\covectors)$.
	
	To prove $\rk(\covectors)\leq\vcd(G)$, we contract any $\Theta$-class $E_i$ and set
	$G'=\pi_i(G)$ and $\covectors'=\covectors\backslash \{ i\}$. By induction hypothesis,
    $\rk(\covectors')\leq\vcd(G')\leq \vcd(G)$.
	Thus, if $\rk(\covectors')=\rk(\covectors)$ or $\vcd(G')=\vcd(G)-1$, then we are obviously
	done. Thus suppose, that for any $\Theta$-class $E_i$ and
	$G'=\pi_i(G)$, we have $\rk(\covectors')=\rk(\covectors)-1$ and $\vcd(G')=\vcd(G)$.
	
	If a $\Theta$-class $E_i$ of $G$ crosses the faces $\F(X)$ of all
	cocircuits $X\in\cocircuits$, then $\covectors$ is not simple. Therefore,
	for any cocircuit $X\in\cocircuits$ there is a $\Theta$-class $E_i$ not
	crossing $\F(X)$. However, since when we contract $E_i$ the rank decreases
	by 1, we conclude that the resulting OM coincides with $F(X)$. Indeed,
	after contraction the rank of $\F(X)$ remains the same. Hence, if $X$ would
	remain a cocircuit the global rank would not decrease. Consequently, $G'$
	is the tope graph of $\F(X)$. Thus, $G$ and $G_i^+=F(X)$ are antipodal
	partial cubes and $G_i^+$ is gated (the latter because it is a face of
	$G$). Since $G$ is antipodal, $G_i^-\cong G_i^+$ is antipodal as well.
	Since we are in a COM, $G_i^-$ is also a gated subgraph of $G$
	by~\cite{KnMa}. Since all $\Theta$-classes of $G_i^+,G_i^-$ coincide, the
	path from any vertex in $G_i^+$ to its gate in $G_i^-$ consists of an edge
	from $E_i$, and vice versa. Thus, $G\cong G_i^+\product K_2$. From the next
	claim we obtain that $G$ must be a cube, contrary to our assumption.
	
		\begin{claim}
			If $G$ is a partial cube and $G\cong G_i^+\product K_2$ for any
			$\Theta$-class $E_i$, then $G$ is a hypercube.
		\end{claim}
		
		\begin{proof}
			First, note that any two $\Theta$-classes $E_i,E_j$ of $G$ must
			\emph{cross}, i.e., $G_i^+\cap G_j^+,G_i^-\cap G_j^+,G_i^+\cap
			G_j^-,G_i^-\cap G_j^-\neq \emptyset$. Indeed, since $G\cong
			G_i^+\product K_2$, after contracting $E_i$ we get a
			graph isomorphic to $G_i^+$ and  $G^-_i$, which has the same
			$\Theta$-classes as $G$ except $E_i$. This implies that any other
			class $E_j$ crosses both $G^+_i$ and $G^-_i$.
			We assert that $G_i^+$ satisfies the hypothesis of the claim. For
			each vertex in the halfspace $G_i^+\cap G_j^+$ of $G^+_i$ defined
			by $E_j$ its unique neighbor with respect to the factorization
			$G\cong G_j^+\product K_2$ is in $G_i^+\cap G_j^-$ and vice versa.
			Therefore, $G_i^+\cong (G_i^+\cap G_j^+)\product K_2$. The same
			holds for $G_i^-$. By induction assumption,  $G_i^+,G_j^+$ are
			hypercubes. Consequently, $G$ is the Cartesian product of a
			hypercube with an edge, whence a hypercube itself.
		\end{proof}
	
	That the VC-dimension of the tope graph of a COM is attained by a face  is proved
	in~\cite[Lemma 42]{ChKnPh}. This also implies the result for AMPs. For
	AMPs, this also follows from the equality $\uX(G)=\oX(G)$. The equality for
	OMs is stated in \cite{KnMa} with a reference
	to~\cite{daS95}.
\end{proof}

\begin{rexample}
The running example $M$ is the COM-amalgam of 12 maximal faces: ten $C_4$, one $C_8$, and one rhombododecahedron (see Fig.
\ref{fig:prism_completion}(b)). By Lemma \ref{VCdimOM}, $\vcd(M)=\max\{
\vcd(C_4),\vcd(C_8),\vcd{D}\}=\max\{2,3\}=3$.
\end{rexample}


\section{Ample completions of OMs}
The goal of this section is to prove the following result :

\begin{theorem} \label{OMtoAMP} Let $\covectors$ be an oriented matroid of rank $d$ and $G$ its tope graph,
which henceforth is of VC-dimension $d$. Then $G$ can be completed to an ample partial cube $\ac(G)$ of VC-dimension~$d$.
\end{theorem}

This completion is done in two steps. First, we use the known
result that any OM can be completed to a UOM of the same rank. Consequently,
the tope graph of any OM can be completed to a tope graph of a UOM of the same
VC-dimension. Second, we recursively complete the tope
graph of any UOM to an ample partial cube of the same VC-dimension.


\begin{example}
The prism $\Pi=C_6 \product P_2$ is the tope graph of an OM and is a proper isometric subgraph of $Q_4$. Contracting any
$\Theta$-class of $\Pi$, except the vertical one, results into $Q_3$, thus $\vcd(\Pi)=3$. The rhombododecahedron $D$ is obtained as a UOM-completion of
$\Pi$.  In Fig. \ref{fig:prism_completion}(c) we present an ample completion of $D$ (and thus of $\Pi$)
obtained as in the proof of Lemma \ref{lem:UOMtoAOM}: first, the $\Theta$-class of vertical edges of $D$  are contracted to obtain the 3-cube $Q_3$.
At the second stage,  an ample completion of $D$ is obtained by performing an ample expansion of $Q_3$ along a $Q^-_3$ ($Q_3$ minus a vertex).
\end{example}

\begin{figure}[htb]
	\centering
	\includegraphics[width=.75\textwidth]{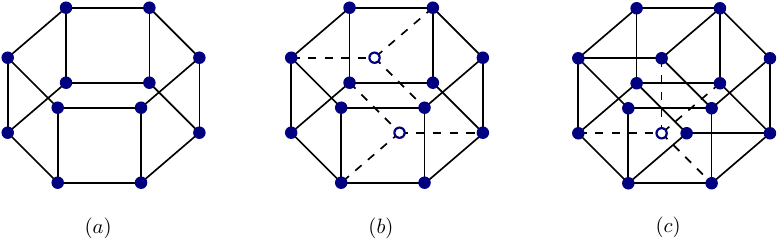}
	\caption{(a) The prism $\Pi$. (b) $D$ as a UOM-completion of $\Pi$. (c) The ample
	completion of $D$ and $\Pi$.}
	\label{fig:prism_completion}
\end{figure}

\subsection{UOMs: uniform oriented matroids}\label{sub:UOM}


In~\cite[Proposition 2.2.4]{BjLVStWhZi} it is stated that the combinatorial
types of cubical zonotopes, i.e., zonotopes in which all proper faces are
cubes~\cite{She}, are in one-to-one correspondence with realizable uniform
matroids (up to reorientation).
A way of generalizing this to general UOMs is basically due to
Lawrence~\cite{La}, also see~\cite[Exercise 3.28]{BjLVStWhZi}. It has been
restated in terms of tope graphs in~\cite{KnMa1}: the tope graphs of UOMs correspond to antipodal
partial cubes in which all proper antipodal subgraphs are cubes. Let us give a
proof.

\begin{lemma}\label{lem:defUOM}
	For the tope graph $G$ of an OM $\covectors$, the following conditions are equivalent:
	\begin{itemize}
		\item[(i)] $G$ is the tope graph of a UOM;
		\item[(ii)] all proper faces, (i.e., all proper antipodal subgraphs,)
		of $G$ are hypercubes;
		\item[(iii)] all halfspaces (equivalently, all half-carriers) of $G$ are ample
		partial cubes.
	\end{itemize}
\end{lemma}

\begin{proof}
	(i)$\Rightarrow$(ii)
	By definition of a UOM of rank $r$, the cocircuits $\mathcal{C}^*$ are exactly
	orientations of sets of support of size $m-r+1$. In an OM the
	$\covectors$ can be obtained from $\cocircuits$ by taking all possible
	compositions. Thus, $\covectors$ consists of all possible sign
	vectors $Y\in\{+,-,0\}^U$ with $X\leq Y$ for some $X\in \cocircuits$. In
	other words, in a UOM we have $\covectors=\uparr \cocircuits=\uparr
	(\covectors\setminus\{\bf{0}\})$. In particular, for every face $\F(Y),
	Y\ne \{\bf{0}\}$ of $\covectors$, we have that $\uparr \F(Y)$  is in
	$\covectors$, thus $\F(Y)$ is a hypercube.
	
	(ii)$\Rightarrow$(iii): Since any OM is a COM, by \cite{BaChKn}, any
	halfspace and any half-carrier of $G$ is the tope graph of a COM. Since all faces of this
	halfspace (respectively, half-carrier) are cubes, this COM satisfies the
	ideal composition axiom (IC) and thus is ample.
	
	(iii)$\Rightarrow$(ii): Suppose that some proper face $\F(X)$ of $G$ is not
	a cube. Then there exists a $\Theta$-class $E_i$ such that $\F(X)$ is
	contained in one of the halfspaces $G^-_i$ or $G^+_i$, say $\F(X)\subseteq
	G^+_i$. Then $\F(X)$ is a face of $G^+_i$, thus $G^+_i$ does not satisfies
	(IC), thus is not ample. 
	
	(ii)$\&$(iii)$\Rightarrow$(i): Let $G$ be the tope graph of an OM $\covectors$ of rank $r$ such that every
	proper face is a hypercube and all halfspaces are AMPs. Since all maximal
	antipodal subgraphs of the halfspace of the tope graph of an OM have the same VC-dimension
	and the VC-dimension of a hypercube is its dimension, all cocircuits $X$ of $\covectors$
	have the same support size. Since $G$ is antipodal
	by Lemma \ref{lem:antipodal}, 
	$\vcd(G)$ is one more than the VC-dimension of $F(X)$. Thus, all $X$ have
	support of size $m-r+1$. This implies, that every set of size $m-r+1$ is
	the support of a cocircuit, since otherwise it has to be in containment
	relation with some cocircuit $X$, which then contradicts the support size
	property.
\end{proof}

\begin{corollary}\label{lem:halfUOM}
	If $G$ is the tope graph of a UOM $\covectors$, $\vcd(G)=d$, and $G'$ is a proper convex subgraph of
	$G$, then $G'$ is ample and $\vcd(G')\le d-1$.
\end{corollary}

\begin{proof}
	By Lemma~ \ref{lem:defUOM}, any halfspace $H$ of $G$ is ample. By~ Lemmas
	\ref{lem:antipodal} and \ref{OM-antipodal}, 
	$\vcd(H)\le d-1$. We are done, since any proper convex subgraph $G'$ of $G$
	is an intersection of halfspaces.
\end{proof}
%

The following lemma is a well-known result in OM theory. We present a proof illustrating our tools.

\begin{lemma}\label{lem:contractionUOM}
	The class of tope graphs of UOMs is closed under contractions.
\end{lemma}

\begin{proof}
	Let $G$ be the tope graph of a UOM and $E_i$ be a $\Theta$-class of $G$. To show that
	$G'=\pi_i(G)$ is the tope graph of a UOM, by Lemma \ref{lem:defUOM} we have to prove that all
	halfspaces of $G'$ are ample partial cubes. Consider a $\Theta$-class
	$E_j\neq E_i$ of $G$. Since $E_j\neq E_i$ there is a corresponding
	$\Theta$-class in $G'$.
	Since $G$ is the tope graph of a UOM, by Lemma \ref{lem:defUOM} the  halfspaces of $G$ are
	ample partial cubes, in particular $G^+_j$ is ample. Moreover, as ample
	partial cubes are closed under contractions, $\pi_i(G^+_j)$ is ample.
	Since halfspaces can be viewed as restrictions and knowing that
	contractions and restrictions commute in partial cubes (see, for example
	\cite{ChKnMa}), we get that $\pi_i(G^+_j)=(\pi_i(G))^+_j=(G')^+_j$ is
	ample. Consequently, the halfspaces of $G'$ are ample.
\end{proof}

%
%

\begin{lemma}\label{lem:expandUOM}
	Let $G'$ be a partial cube obtained from the tope graph $G$ of a UOM $\covectors$  by an isometric
	expansion with respect to $(G^1,G^0,G^2)$ such that $G^1=-G^2$, $G^0$ is an
	isometric subgraph of $G$, and $G^0$ is the tope graph of a UOM. Then $G'$ is the tope graph of a UOM. If
	$\vcd(G)=d$ and $\vcd(G^0)\le d-1$, then $\vcd(G')\le d$.
\end{lemma}

\begin{proof}
	Since $G^1=-G^2$, the graph $G'$ is antipodal, see e.g.~\cite[Lemma
	2.14]{KnMa}. By Lemma \ref{lem:defUOM}, to prove that $G'$ is the tope graph of a UOM, we
	show that all antipodal subgraphs of $G'$ are cubes. Let $A'$ be an
	antipodal subgraph of $G'$ and let $E_i$ be the unique $\Theta$-class of
	$G'$ which does not exist in $G$, i.e., $\pi_i(G')=G$.  If $A'$ does not
	use the $\Theta$-class $E_i$, then $\pi_i(A')=A'$ is a subgraph of
	$\pi_i(G')=G$, thus $A'$ is an antipodal subgraph of $G$. As $G$ is a tope
	graph of a UOM,
	by Lemma \ref{lem:defUOM}, $A'$ is a cube. Otherwise, suppose that $A'$
	uses the $\Theta$-class $E_i$. By Lemma~\ref{lem:antipodal_clos_pi},
	$A=\pi_i(A')$ is an antipodal subgraph of $G$. Since $G$ is a tope graph of
	a UOM, using
	Lemma \ref{lem:defUOM}, $A$ is a cube $\Q_k$ in $G$. Moreover, $A'$ can be
	viewed as an isometric expansion $(A^1,A^0,A^2)$ of $A=\Q_k$ with
	$A^1=-A^2$.
	Moreover, since $G^0$ is an isometric subgraph of $G$, $A^0$ is a convex
	subgraph of $G^0$  that is closed under antipodes. Thus, $A^0$ is an
	antipodal subgraph of $G^0$. Finally, since $A'$ is a proper subgraph of
	$G$, $A^0$ is a proper subgraph of $G^0$. Thus, $A^0$ is a cube since $G^0$
	is a tope graph of a UOM. Thus, by the properties of isometric expansions,
	$G^0\cap
	\Q_k=\Q_k$ and $A'=\Q_{k+1}$ is a cube. The statement about the
	VC-dimension follows straightforward from Lemma~\ref{expansion-Qd+1}.
\end{proof}


\subsection{Completions of tope graphs of OMs to tope graphs of UOMs}

Now, we use standard OM theory to obtain:
\begin{lemma}\label{lem:OMtoUOM}
	The tope graph $G$ of any OM $\covectors$ can be completed to the tope graph of a UOM of the same VC-dimension.
\end{lemma}

\begin{proof}
	By~\cite[Definition 7.7.6]{BjLVStWhZi},~\cite[Proposition
	7.7.5]{BjLVStWhZi}, and some easy translation from topes to tope graphs
	there is a \emph{weak map} from the tope graph $G_1$ of an OM to a tope graph $G_2$ of an OM, if $G_2$ is a
	subgraph of $G_1$, both are isometric subgraphs of the same hypercube, and
	both have the same isometric dimension. This implies, that $G_2$ is an
	isometric subgraph of $G_1$. Now~\cite[Corollary 7.7.9]{BjLVStWhZi} says
	that every tope graph $G_2$ of an OM  is the weak map image of  the tope graph $G_1$ of a UOM of the same rank,
	i.e., the same VC-dimension.
\end{proof}

\subsection{Ample completions of tope graphs of UOMs}

\begin{lemma}\label{lem:expAmple}
	A peripheral expansion $G'$ of an ample partial cube $G$
	with respect to an ample subgraph $H$ is ample. 
\end{lemma}

\begin{proof}
	Clearly, $H':=H\product K_2$ is ample. Then $G'$ is an AMP-amalgam of $G$ and $H'$ along
	$H$. Since $G'$ is a partial cube (as an isometric expansion of $G$), by
	Proposition \ref{AMPamalgam} $G'$ is ample.
\end{proof}


\begin{lemma}\label{lem:UOMtoAOM}
 	If $G$ is the tope graph of a UOM of rank $d$, then $G$ can be completed in $\C(G)$ to
 	an ample partial cube $\amp(G)$ of VC-dimension $d$. 
\end{lemma}

\begin{proof} Let $E_i$ be any $\Theta$-class of $G$ and let $G^+_i$ and
$G^-_i$ be the halfspaces defined by $E_i$. By Lemma \ref{lem:defUOM},
$G^+_i$ and $G^-_i$ are ample partial cubes. Let $G'=\pi_i(G)$ be the partial
cube obtained by contracting the edges of $E_i$. By Lemma
\ref{lem:contractionUOM}, $G'$ is a tope graph of a UOM. Since $\pi_i(G^+_i)$
and $\pi_i(G^-_i)$
are isomorphic to $G^+_i$ and $G^-_i$, respectively, those subgraphs of $G'$
are ample partial cubes. 
By Corollary \ref{lem:halfUOM}, $G^+_i,G^-_i$ and  $\pi_i(G^+_i), \pi_i(G^-_i)$ have VC-dimension at most $d-1$.

By induction hypothesis, $G'$ admits an ample completion $\amp(G')$ included in
$\C(G')$ (where  $\C(G')$ is considered in  the hypercube of one less
dimension) and having VC-dimension $d$. Define $\amp(G)$ as the peripheral
expansion of $\amp(G')$ with respect to the ample partial cube $\pi_i(G^+_i)$.
By Lemma \ref{lem:expAmple}, $\amp(G)$ is indeed ample. Notice also that
$\amp(G)$ is contained in $\C(G)$ (considered in the original hypercube). It
remains to show that  $\amp(G)$  has VC-dimension $d$.
The partial cube $\amp(G)$ is obtained from $\amp(G')$ by an isometric peripheral expansion  with respect to $\pi_i(G^+_i)$ of VC-dimension $\le d-1$. By Proposition \ref{expansion-Qd+1}, $\amp(G)$ has VC-dimension $d$.
\end{proof}

%

This concludes the proof of Theorem \ref{OMtoAMP}.

\section{Ample completions of CUOMs}

Recall that a  COM $\covectors$ is called a {\it complex of uniform oriented matroids}
(CUOM) if each facet of $\covectors$ is a UOM. The goal of this section is to prove the
following result:

\begin{theorem} \label{CUOMtoAMP} Let $\covectors$ be a complex of uniform oriented matroids and $G$ its tope graph
of VC-dimension $d$. Then $G$ can be completed to an ample partial cube $\ac(G)$ of VC-dimension~$d$.
\end{theorem}

\begin{remark} Note that in a COM of VC-dimension $2$ the faces correspond to
vertices, edges, and even cycles in its tope graph. Hence, 2-dimensional COMs are CUOMs and
Theorem~\ref{CUOMtoAMP} generalizes the ample completion of 2-dimensional COMs
presented in \cite[Subsection 6.2]{ChKnPh}.
\end{remark}

The idea of the proof is to independently complete the facets of $G$ to AMPs
(using the recursive completion of tope graphs of UOMs) and show that their union is ample and has VC-dimension $d$.


\begin{example}
	\label{ex:CUOM}
	Fig. \ref{fig:CUOM_completion} presents the ample completion of the tope graph of a CUOM with two
	$D$-facets and one $Q_3$-facet. It is obtained by completing the two rhombododecahedra as UOMs (see Fig. \ref{fig:prism_completion} (b)\& (c)).
\end{example}

\begin{figure}[h]\centering
	\includegraphics[width=0.75\linewidth]{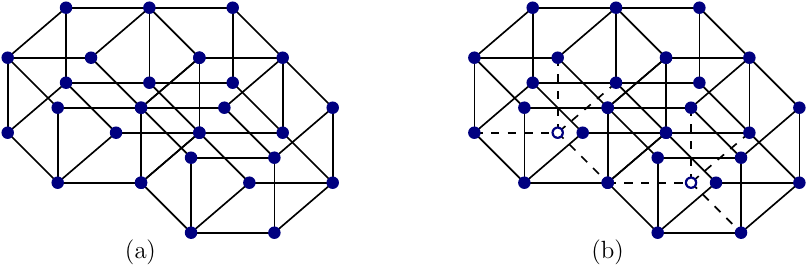}
	\caption{(a) The tope graph of a CUOM. (b) Its ample completion.}
	\label{fig:CUOM_completion}
\end{figure}


\subsection{A characterization of CUOMs} \label{characterization-CUOMs}

We start with a characterization of CUOMs: 

\begin{proposition} \label{CUOM}
	For the tope graph $G$ of a COM $\covectors$ the following conditions are equivalent:
	\begin{itemize}
		\item[(i)] $G$ is the tope graph of a CUOM;
		\item[(ii)] all non inclusion maximal faces of $G$ are hypercubes;
		\item[(iii)]  all half-carriers of $G$ are ample partial cubes.
	\end{itemize}
\end{proposition}

\begin{proof}
	(i)$\Rightarrow$(ii): This trivially follows from the definitions of UOMs
	and CUOMs.
	
	(ii)$\Rightarrow$(iii): From Proposition \ref{carriers} it follows that the
	half-carriers $N^+_i(G)$ and $N^-_i(G)$ of the tope graph of a COM $G$ are tope graphs of
    COMs. By the
	definition of half-carriers, each face $\F(Y)$ of a half-carrier, say of
	$N^+_i(G)$, is properly contained in a facet $\F(X)$ of the carrier
	$N_i(G)$. Then $\F(X)$  is a facet of $G$. Thus $\F(X)$ is a tope graph of a UOM and
	$\F(Y)$ is a cube. The tope graph of a COM in which all faces are cubes is ample because it
	satisfies (IC). This proves that all half-carriers of $G$ are ample.
	
	(iii)$\Rightarrow$(i): Suppose $\covectors$ is not a CUOM, i.e., its tope graph $G$ contains a
	facet $\F(X)$ which is not the tope graph of a UOM. By Lemma \ref{lem:defUOM}(iii), $\F(X)$
	contains a non-ample half-carrier, say $N^+_i(\F(X))$ defined by the
	$\Theta$-class $E'_i$ of $\F(X)$. This $\Theta$-class $E'_i$ can be
	extended to a $\Theta$-class $E_i$ of $G$ and $N^+_i(\F(X))$ is included in
	the half-carrier $N^+_i(G)$ of $G$. Since $N^+_i(\F(X))$ is not ample,
	$N^+_i(G)$ is also not ample.
\end{proof}

\subsection{Single gated extensions of partial cubes} \label{extension-partial-cubes}
We mentioned already that all faces of a tope graph $G$ of a COM  are gated subgraphs of $G$ and
the completion method of $G$  should first take care of completing  its faces.
In this subsection, we prove a general result about a partial completion of a
partial cube $G$ with respect to a gated subgraph. We suppose that $G$ is
isometrically embedded in the hypercube $\Q_m=\Q(U)$. Recall that $\C(G)$ is
the smallest cube of $\Q_m$ containing $G$.

\begin{proposition}
	\label{thm:gated_set_completion}
	Let $G$ be a partial cube and $H$ be a gated subgraph of $G$. Let $H'$ be
	an isometric subgraph of $\Q_m$ such that $H\subseteq H'\subseteq \C(H)$
	and let $G'$ be the subgraph of $\Q_m$ induced by $V(G)\cup V(H')$.
	Then the following holds:
	\begin{itemize}
		\item[(i)] $G'$ is an isometric subgraph of $\Q_m$;
		\item[(ii)] $H'$ is a gated subgraph of $G'$ and for each vertex $v$
		its gates in $H$ and $H'$ coincide;
		\item[(iii)] $d:=\vcd(G') = \max\{\vcd(G), \vcd(H')\}$.
	\end{itemize}
	In particular, if $\vcd(H')\le \vcd(G)$, then $\vcd(G')=d$.
\end{proposition}

\begin{proof}
	Since $H$ is a gated and thus a convex subgraph of $G$, we have
	$\C(H)\cap V(G)=V(H)$. First we prove that $G'$ is an isometric subgraph of
	$\Q_m$. Since $G$ and $H'$ are isometric subgraphs of $\Q_m$, it
	suffices to show that any vertex $v \in V(G) \setminus V(H)$ and any vertex
	$u \in V(H') \setminus V(H)$ can be connected in the graph $G'$ by a
	shortest path of $\Q_m$. Since $H$ is a gated subgraph of $G$, let $v'$ be
	the gate of $v$ in $H$. Let $P$ be any shortest $(v,v')$-path of $G$.
	Since $v'$ is the gate of $v$ in $H$, by Lemma \ref{gated-theta}, $P$ does
	not use any $\Theta$-class that appear in $H$.  From the definition of
	$\C(H)$, the $\Theta$-classes of $H$ and $\C(H)$ coincide. Since $H'$ is an
	isometric subgraph of $\C(H)$ and $v' \in V(H), u\in V(H')$, any shortest
	$(v',u)$-path $S$ of $G'$ can use only the $\Theta$-classes of $\C(H)$, and
	thus of $H$. This implies that the concatenation of $P$ and $S$ is a
	$(v,u)$-path $R$ of $G'$ whose all $\Theta$-classes are pairwise distinct.
	By Lemma \ref{path-theta}, $R$ is a shortest path of $\Q_m$, establishing
	that $G'$ is an isometric subgraph of $\Q_m$.
	Moreover, the gate of $v$ in $H'$ is also $v'$, because from $v'$ (the gate
	of $v$ in $H$) we can reach any vertex of $H'$ using only $\Theta$-classes
	belonging to $H$. We conclude that the gates of $H'$ coincide with those of
	$H$. This proves the assertions (i) and (ii).
	
	Before proving assertion (iii), we establish the following claim:
	
	\begin{claim} \label{traversingV(H)}
		All shortest paths of $G'$ from a vertex of $v\in V(G) \setminus V(H)$
		to a vertex of $z\in V(H') \setminus V(H)$ traverse $V(H)$.
	\end{claim}
	
	\begin{proof}
		Suppose by way of contradiction that there exists a shortest
		$(v,z)$-path $T$ of $G'$ not intersecting $V(H)$.  Since $v\in V(G)
		\setminus V(H)$,   $z\in V(H') \setminus V(H)$, and $T\subset (V(G)\cup
		V(H'))\setminus V(H)$, the path $T$ contains an edge $xy$ with $x\in
		V(G) \setminus V(H)$ and $y\in V(H') \setminus V(H)$. We proved above
		that for any vertex of $G$ its gates in $H$ and in $H'$ are the same.
		Since the vertices $x\in V(G) \setminus V(H)$ and $y\in V(H')$ are
		adjacent, $y$ must be the gate of $x$ in $H'$. Thus $y$ is the gate of
		$x$ in $H$, contrary to the assumption that $y\notin V(H)$.
	\end{proof}
	
	To prove (iii), suppose by way of contradiction that 
	$d> \max\{\vcd(G), \vcd(H')\}$. This implies that $G'$ shatters the
	$d$-cube $Q_d:=\Q(X)$ for some $X\subseteq U$, $|X|=d$. By Lemma
	\ref{shattering-fibers}, each fiber $G'_{X'}, X'\subseteq X$ of $G'$ is
	nonempty. Let $\psi: V(G')\rightarrow V(\Q_d)$ be the shattering map,
	mapping each $X$-fiber $G'_{X'}$ of $G'$ to the subset $X'$ of $X$. Since
	$d>\max\{\vcd(G), \vcd(H')\}$, neither $G$ nor $H'$ shatter $\Q_d$,
	therefore the map $\psi$ restricted to $V(G)$ and to $V(H')$ is no longer
	shattering. By Lemma \ref{shattering-fibers}, there exist two subsets
	$Y,Z$ of $X$ such that the fibers $G_Y$ and $H'_Z$ (of $G$ and $H'$,
	respectively) are empty. On the other hand, the fibers $G'_Y$ and $G'_Z$
	are nonempty.
	
	
	By Claim \ref{traversingV(H)} all shortest paths from $V(G) \setminus V(H)$
	to $V(H') \setminus V(H)$ pass through $V(H)$ and since $V(H) \subseteq
	V(H')\cap V(G)$, whence all vertices of the fiber $G'_Y$ are included in
	$V(H')\setminus V(G)$. This implies that every $X$-edge of $G'$ with one
	end in $G'_Y$ must have the other end in $H'$. Since $H'$ has the same
	$\Theta$-classes as $H$, each such $X$-edge is defined by a $\Theta$-class
	of $H$. Since in $\Q_d$ any vertex is incident to an edge from every
	$\Theta$-class, $G'_Y$ must be incident to all types of $X$-edges.
	
	Now, applying again Claim \ref{traversingV(H)}, we conclude that the fiber
	$G'_Z$ is included in $V(G')\setminus V(H')$. Pick any vertex $v\in
	V(G'_Z)$ and let $v'$ be its gate in $H$ (and in $H'$). Since $v\in
	V(G')\setminus V(H')$, necessarily $v'\ne v$. Let $P$ be a shortest
	$(v,v')$-path of $G'$. Since $v$ and $v'$ belong to different fibers of
	$G'$, necessarily $P$ contains an $X$-edge $xy$. Since any $X$-edge is
	defined by a $\Theta$-class of $H$ (and $H')$, Lemma \ref{gated-theta}
	yields a contradiction with the assertion (ii) that $v'$ is the gate of $v$
	in $H$ and $H'$. This contradiction shows that $d=\max\{\vcd(G),\vcd(H')\}$.
\end{proof}

\begin{remark}
	If $G'$ is obtained from a partial cube $G$ via a single extension with
	respect to a gated subgraph $H$  (as in Lemma
	\ref{thm:gated_set_completion}), some gated subgraphs of $G$ may no
	longer be gated in $G'$. Next we  show that this phenomenon does not arise in
	tope graphs of CUOMs.
\end{remark}

\subsection{Mutual projections between faces of COMs} \label{mutual-projections}
In the proof of Theorem \ref{CUOMtoAMP} we use the following result of Dress and Scharlau \cite{DrSch} on mutual metric projections  between gated sets.  
Recall that the {\it distance} $d(A,B)$ between two sets of vertices $A,B$ of a graph $G$ is $\min \{ d(a,b): a\in A, b\in B\}$. 
The {\it metric projection} $\pr_B(A)$ of $B$ on $A$ consists of all vertices
$a$ of $A$ realizing the distance $d(A,B)$ between $A$ and $B$, i.e.,
$\pr_B(A)=\{ a\in A: d(a,B)=d(A,B)\}$. 

\begin{theorem}
	\label{mutual-gates} \cite[Theorem]{DrSch}
	Let $A$ and $B$ be two gated subgraphs of a graph $G$. Then $\pr_A(B)$ and
	$\pr_B(A)$ induce two isomorphic gated subgraphs of $G$ such that for any
	vertex $a'\in \pr_B(A)$ if $b'=\pr_{a'}(B)$, then
	$d(a',b')=d(\pr_A(B),\pr_B(A))=d(A,B)$, $\pr_{b'}(A)=a'$, and the map
	$a'\mapsto b'$ defines an isomorphism between $\pr_A(B)$ and $\pr_B(A)$.
\end{theorem}




For $X,Y\in \covectors$,  we denote by $\pr_{\F(X)}(\F(Y))$ the metric
projection of $\F(X)$ on $\F(Y)$ in the tope graph $G$ of $\covectors$ and by $\pr_{\C(X)}(\C(Y))$
the metric projection of the cube $\C(X)$ on the cube $\C(Y)$ in the hypercube
$\Q(U)$.  Since by Lemma \ref{face-gated} the faces $\F(X)$ of $X\in
\covectors$ are  gated in $G$ and all cubes $\C(X)$ are gated in  $\Q(U)$,
applying Theorem \ref{mutual-gates} to them we conclude that
$\pr_{\F(X)}(\F(Y))$ and $\pr_{\F(Y)}(\F(X))$ are isomorphic as well as
$\pr_{\C(X)}(\C(Y))$ and $\pr_{\C(Y)}(\C(X))$ and those isomorphisms map the
pairs of vertices realizing the distances between $\pr_{\F(X)}(\F(Y))$ and
$\pr_{\F(Y)}(\F(X))$ and between $\pr_{\C(X)}(\C(Y))$ and $\pr_{\C(Y)}(\C(X))$.
We say that two faces $\F(X)$ and $\F(Y)$ of $\covectors$ are {\it parallel} if
$\pr_{\F(X)}\F(Y)=\F(Y)$ and  $\pr_{\F(Y)}\F(X)=\F(X)$. A {\it gallery} between
two parallel faces $\F(X)$ and $\F(Y)$ of  $\covectors$ is a sequence of faces
$(\F(X)=\F(X_0),\F(X_1),\ldots, \F(X_{k-1}),\F(X_k)=\F(Y))$ such that any two
faces of this sequence are parallel and any two consecutive faces
$\F(X_{i-1}),\F(X_i)$ are facets of a common face of $\covectors$. A {\it
geodesic gallery} between $\F(X)$ and $\F(Y)$ is a gallery of length
$d(\F(X),\F(Y))=|S(X,Y)|$. Two parallel faces $\F(X),\F(Y)$ are called {\it
adjacent} if $|S(X,Y)|=1$, i.e., $\F(X)$ and $\F(Y)$ are opposite facets of a
face of $\covectors$. See Fig. \ref{fig:galerie} for an illustration.

\begin{figure}[htb]
	\centering
	\includegraphics[width=.5\textwidth]{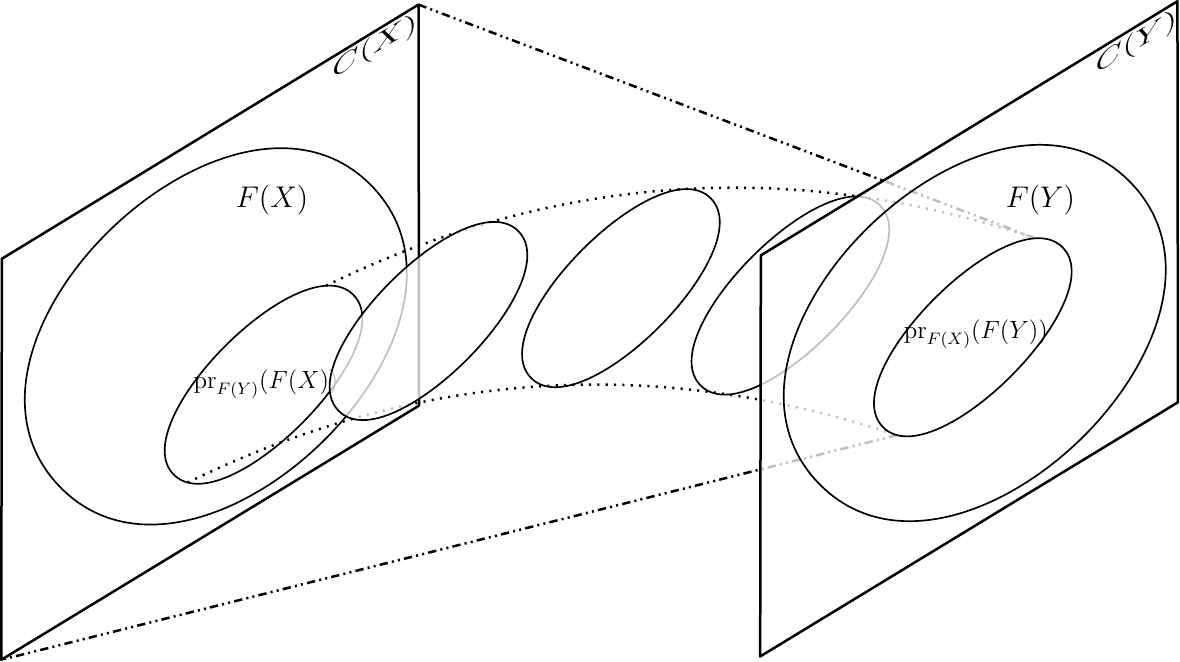}
	\caption{Two faces $F(X)$ and $F(Y)$, their mutual projections
	$\pr_{F(Y)}(F(X))$ and $\pr_{F(X)}(F(Y))$, and a  geodesic 
    gallery connecting them.}
	\label{fig:galerie}
\end{figure}

The most part of next result holds for all COMs.
Therefore, we specify in the assertions where we require CUOMs. We use
simultaneously the covector and the tope graph notations.

\begin{proposition}
	\label{prop:projection}
	For any two covectors $X,Y$ of a COM $\covectors$,  the following
	properties hold:
	\begin{itemize}
		\item[(i)] $d(\F(X),\F(Y))=d(\C(X),\C(Y))=|S(X,Y)|$;
		\item[(ii)]  $\pr_{\F(X)}(\F(Y))\subseteq \pr_{\C(X)}(\C(Y))$ and
		$\pr_{\F(Y)}(\F(X))\subseteq \pr_{\C(Y)}(\C(X))$;
		\item[(iii)] $\pr_{\F(Y)}(\F(X))=\F(X\circ Y)$ and
		$\pr_{\F(X)}(\F(Y))=\F(Y\circ X)$;
		\item[(iv)] $\F(X)$ and $\F(Y)$ are parallel if and only if
		$\underline{X}=\underline{Y}$ (or, equivalently, if $X^0=Y^0$);
		\item[(v)] $\pr_{\F(Y)}(\F(X))$ and $\pr_{\F(X)}(\F(Y))$ are parallel
		faces of $\covectors$;
		\item[(vi)]  any two parallel faces $\F(X)$ and $\F(Y)$  can be
		connected in $\covectors$ by a geodesic gallery;
		\item[(vii)] if $\F(X)$ is a facet of $\covectors$, then
		$\pr_{\F(Y)}(\F(X))$ is a proper face of $\F(X)$;
		\item[(viii)] if $\covectors$ is a CUOM and $\F(X),\F(Y)$ are facets,
		then  $\pr_{\F(Y)}(\F(X)),$ $\pr_{\F(X)}(\F(Y))$ are cubes;
		\item[(ix)] if $\covectors$ is a CUOM and $\F(X),\F(Y)$ are facets,
		then $\pr_{\F(X)}(\F(Y))=\pr_{\C(X)}(\C(Y))$ and
		$\pr_{\F(Y)}(\F(X))=\pr_{\C(Y)}(\C(X))$.
	\end{itemize}
\end{proposition}

\begin{proof}
	(i): From the definition of $\C(X)$ and $\C(Y)$ it follows that
	$\F(X)$ and $\C(X)$ have the same $\Theta$-classes and $\F(Y)$ and $\C(Y)$
	have the same $\Theta$-classes. Therefore the set of $\Theta$-classes
	separating the faces $\F(X)$ and $\F(Y)$ is the same as the set of
	$\Theta$-classes separating the cubes $\C(X)$ and $\C(Y)$ and coincides
	with $S(X,Y)$. Therefore, $d(\F(X),\F(Y))=d(\C(X),\C(Y))=|S(X,Y)|$.

	(ii): $\pr_{\F(X)}(\F(Y))\subseteq \pr_{\C(X)}(\C(Y))$ and
	$\pr_{\F(Y)}(\F(X))\subseteq \pr_{\C(Y)}(\C(X))$ follow from  (i).
	
	(iii): Note that for any two covectors $X$ and $Y$,
	$\underline{X\circ Y}=\underline{Y\circ X}$ and $S(X,Y)=S(X\circ Y,Y\circ
	X)$ hold. Since
	$d(\F(X\circ Y),\F(Y\circ X))=|S(X\circ Y,Y\circ X)|$ and $S(X\circ
	Y,Y\circ
	X)=S(X,Y)$ from property (i) we obtain that $d(\F(X),\F(Y))=d(\F(X\circ
	Y),\F(Y\circ X))$, thus $\F(X\circ Y)\subseteq \pr_{\F(Y)}(\F(X))$ and
	$\F(Y\circ X)\subseteq \pr_{\F(X)}(\F(Y))$. To prove the converse
	inclusions, suppose by way of contradiction that there exists a tope $Z\in
	\pr_{\F(Y)}(\F(X))\setminus \F(X\circ Y)$. Since $\pr_{\F(Y)}(\F(X))$ is
	gated,
	we can suppose that $Z$ is adjacent to a tope $Z'$ of $\F(X\circ Y)$. Let
	$e$
	be the element (a $\Theta$-class) on which $Z$ and $Z'$ differ, say
	$Z_e=+1$
	and $Z'_e=-1$. Since $X\le Z$ and $X\le Z'$, this implies that $X_e=0$. If
	$Y_e=0$, this would imply that $(X\circ Y)_e=0$, thus $Z$ would belong to
	$\F(X\circ Y)$, contrary to our choice of $Z$. Thus $Y_e=-1$. This implies
	that
	$d(Z,Y')\ge |S(X,Y)|+1$ for any tope $Y'\in \F(Y)$. Indeed, $Y'_e=-1$ and
	$Y'_f=-Z_f$ for any $f\in S(X,Y)$. This contradiction shows that
	$\pr_{\F(Y)}(\F(X))=\F(X\circ Y)$ and $\pr_{\F(X)}(\F(Y))=\F(Y\circ X)$.

	(iv): In view of (iii), we can rephrase  the definition of
	parallel
	faces as follows: $\F(X)$ and $\F(Y)$ are parallel if and only if
	$\F(X)=\F(X\circ Y)$ and $\F(Y)=\F(Y\circ X)$, i.e., if and only if
	$X=X\circ
	Y$ and $Y=Y\circ X$. Then one can easily see that  $X=X\circ Y$ and
	$Y=Y\circ
	X$ hold if and only if $\underline{X}=\underline{Y}$ holds.


	(v): This property follows from properties (iii) and (iv).
	
	(vi): Let $\F(X)$ and $\F(Y)$ be two parallel faces. By  (iv),
	$\underline{X}=\underline{Y}$. We proceed by induction on $k:=|S(X,Y)|$.
	Let $B:=\underline{X}=\underline{Y}$. Set $A:=U\setminus B$  and consider
	the COM $(B,\covectors\setminus A)$. Then $X':=X\setminus A$ and
	$Y':=Y\setminus A$ are topes of $\covectors\setminus A$. Note also that the
	distances between $X'$ and $Y'$ and between $X$ and $Y$ are equal to $k$.
	Since the tope graph $G(\covectors\setminus A)$ of the COM
	$\covectors\setminus A$ is an isometric subgraph of the cube $\{-1,+1\}^B$,
	$X'$ and $Y'$ can be connected in $G(\covectors\setminus A)$  by a shortest
	path  of $\{-1,+1\}^{B}$, i.e., by a path of length $k$. Let $Z'$ be the
	neighbor of $X'$ in this path. Then there exists $e\in S(X,Y)=S(X',Y')$
	such that $S(X',Z')=\{e\}$ and $S(Z',Y')=S(X,Y)\setminus \{ e\}$. By the
	definition of $\covectors\setminus A$, there exists a covector $Z\in
	\covectors$ such that $(Z\setminus A)_f=Z'_f$ for each $f\in B$. Hence $Z$
	contains $B$ in its support. Moreover, since
	$\underline{X}=\underline{Y}=B$, $S(X,Z)=\{ e\}$ and
	$S(Z,Y)=S(X,Y)\setminus \{ e\}$. In particular, $Z_f=X_f\ne 0$ for any
	$f\in B\setminus \{ e\}$. Applying the axiom  (SE) to $X,Z$ and $e\in
	S(X,Z)$, we will find $X'\in \covectors$  such that  $X'_e=0$  and
	$X'_f=(X\circ Z)_f$ for all $f\in U\setminus S(X,Z)$. Since
	$\underline{X}=\underline{Y}$ and $S(X,Z)=\{ e\}$, we conclude that
	$X'_f=X_f$ for any $f\in U\setminus  \{ e\}$.
	Consequently, $X'\le X$, i.e., $\F(X)$ is a face of $\F(X')$. Since
	$S(X,Z)=\{ e\}$, $\F(X)$ is a facet of $\F(X')$. By face symmetry (FS),
	$X'':=X'\circ(-X) \in  \covectors$. Notice that $\F(X'')$ is a facet of
	$\F(X')$ symmetric to $\F(X)$, i.e., $\F(X)$ and $\F(X'')$ are adjacent
	parallel faces. Notice also that
	$\underline{X}''=\underline{X}=\underline{Y}$ and, since $X''_e=-X_e$, that
	$S(X'',Y)=S(X,Y)\setminus \{ e\}$. By induction hypothesis, the parallel
	faces $\F(X'')$ and $\F(Y)$ can be connected in $\covectors$ by a geodesic
	gallery. Adding to this gallery the face $\F(X')$, we obtain a geodesic
	gallery connecting $\F(X)$ and $\F(Y)$.

	

	(vii): This property follows from property (vi).
	
	(viii): By (vii) and Proposition \ref{CUOM},  $\pr_{\F(Y)}(\F(X))$
	is a cube as a proper face of $\F(X)$.
	
	(ix):  By (viii),  $\pr_{\F(Y)}(\F(X))$ is a cube  and by
	(iii),   $\pr_{\F(Y)}(\F(X))=\F(X\circ Y)$. By (ii) this cube $\F(X\circ
	Y)$ is  included in the cube  $\pr_{\C(Y)}(\C(X))$.  Suppose 
	that this inclusion is proper. Let $e$ be an element (a $\Theta$-class) of
	the support of $\pr_{\C(Y)}(\C(X))$ which does not belong to the support of
	$\F(X\circ Y)$. Suppose without loss of generality that $Z_e=+1$ for all
	$Z\in \F(X\circ Y)$, i.e., all topes $Z$ of $\F(X\circ Y)$ belong to the
	halfspace $G^+_e$ of $G$.  From the definition of the cubes $\C(X)$ and
	$\C(Y)$,  we conclude that the halfspace $G^-_e$ of $G$ must contains a
	tope $X'$ of $\F(X)$ and a tope $Y'$ of $\F(Y)$. From the definition of the
	mutual gates, we must have a shortest path in $G$ from $X'$ to $Y'$
	traversing via a tope of $\pr_{\F(Y)}(\F(X))$ and a tope of
	$\pr_{\F(X)}(\F(Y))$. But this is impossible because $X',Y'$ belong to
	$G^-_e$ while all the topes of  $\pr_{\F(Y)}(\F(X))=\F(X\circ Y)$  are
	included in $G^+_e$ and $G^-_e$ and $G^+_e$ are convex because
	$G(\covectors)$ is a partial cube.
\end{proof}

\subsection{Proof of Theorem \ref{CUOMtoAMP}}\label{proofCUOMtoAMP}

Let $\covectors$  be a CUOM and $G$ be its tope graph. Let
$\F_1=\F(X_1),\dots,\F_n=\F(X_n)$ be the facets of $\covectors$. Each $\F_i$ is
a UOM and let $\amp(\F_i)$ be the ample completion of $G(\F_i)$ obtained by
Lemma \ref{lem:UOMtoAOM} ($\amp(\F_i)$ is contained in $\C(\F_i)$). Let
$G^*_i=\amp(\F_1)\cup\dots\cup \amp(\F_i)\cup \F_{i+1}\cup\dots\cup \F_n$; in
words, $G^*_i$ is obtained from $G$ by replacing the first $i$ faces
$\F_1,\ldots,\F_i$ by their ample completions $\amp(\F_1),\ldots,\amp(\F_i)$.
Finally, set $G^*:=G^*_n$.
We assert that $G^*$ is ample. For this we will use the amalgamation
results for COMs and  ample partial cubes and Proposition
\ref{thm:gated_set_completion} about single gated set extensions of partial
cubes. Proposition \ref{thm:gated_set_completion} will ensure that each partial
completion $G^*_i$ is a partial cube and its VC-dimension does not increase. As
a result, the final graph $G^*$ is a partial cube and has VC-dimension $d$.

To apply Proposition \ref{thm:gated_set_completion} to each  $G^*_i$, we need that each not yet completed cell $\F_{i+1},\dots ,\F_n$ of $G$ remains
gated in $G^*_1,\ldots, G^*_i$. By Proposition \ref{prop:projection}(ix),
independently in which order the faces $\F_i=\F(X_i)$ and $\F_j=\F(X_j)$ are
completed ($\F_i$ before $\F_j$ or $\F_j$ before $\F_i$), the mutual
projections of $\F_i$ and $\F_j$ initially coincide with those of the cubes
$\C(X_i)$ and $\C(X_j)$, the gate of any vertex $Z\in \amp(\F_i)$ in $\F_j$ (or
of any vertex $Z\in \amp(\F_j)$ in $\F_i$) in each occurring partial completion
will coincide with the gate of $Z$ in the cube $\C(X_j)$ (respectively, with
the gate of $Z$ in the cube $\C(X_i)$). Hence, each partially completed graph
$G^*_i$ is a partial cube and all remaining faces $\F_{i+1},\dots,\F_n$
are gated in $G^*_i$. Thus we can apply Proposition
\ref{thm:gated_set_completion} to the partial cube $G^*_i$ and the remaining
faces  $\F_{i+1},\ldots,F_n$.

Now we show that any edge $uv$ of $G^*$ is included in some completion
$\amp(\F_i)$ of a facet $\F_i$ of $G$. Suppose $u\in \amp(\F_i)$ and $v\in
\amp(\F_j)$. By construction, $\amp(\F_i)\subseteq \C(\F_i)$ and
$\amp(\F_j)\subseteq \C(\F_j)$. Therefore, if $u$ and $v$ are adjacent,
necessarily one of the vertices $u,v$, say $v$, belongs to $\C(\F_i)\cap
\C(\F_j)$. Since $G$ is a tope graph of a CUOM, $\C(\F_i)\cap \C(\F_j)$ is a
proper (cube )face
of $\F_i$ and of $\F_j$. Consequently, $u\in \amp(F_i)$ and $v\in F_i$, and we
are done.

To show that $G^*$ is ample, we use induction on the number of faces of $G$ and
the amalgamation procedures for COMs and ample partial cubes, see
Propositions \ref{amalgamCOM} and \ref{AMPamalgam}. If $G$ consists of a
single maximal face, then we are done by Lemma \ref{lem:UOMtoAOM}. Otherwise,
by Proposition \ref{amalgamCOM} $\covectors$ is a COM-amalgam of two COMs
$\covectors'$ and $\covectors''$ with tope graphs $G'$ and $G''$ such that (1)
every facet of $G$ is a facet of $G'$ or of $G''$ and (2) their intersection
$G_0=G'\cap G''$ is a the tope graph of the  COM $\covectors'\cap
\covectors''$. This implies that $G',G'',$ and $G_0$ are tope graphs of CUOMs: each facet  of
each of them is either (a) a facet of $G$, and thus is the tope graph of a CUOM, or (b) is a
proper face of $G$, and thus is a cube. We call the facets of type (a)
\emph{original facets} and the facets of type (b) \emph{cube facets}.

Let $(G')^*$ be the union of all cube facets of $G'$ and of the ample
completions $\amp(\F_i)$ of all original facets $F_i$ of the tope graph of a
CUOM $G'$. Clearly,
$(G')^*$ is obtained by the completion method described above and applied to
the facets of $G'$.  Analogously, we define the completions $(G'')^*$ and
$(G_0)^*$ of $G''$ and $G_0$, respectively. Since $G',G'',$ and $G_0$ are tope graphs of CUOMs
with less vertices than $G$, by induction hypothesis,  $(G')^*, (G'')^*$ and
$(G_0)^*$ are ample completions of $G',G'',$ and $G_0$, respectively.
Moreover, since each facet of $G$ is a facet of at least one of $G'$ and $G''$,
by the construction and by  what has been proved above, the vertex-set and the
edge-set of the partial cube $G^*$ is the union of the vertex-sets and the
edge-sets of ample partial cubes $(G')^*$ and  $(G'')^*$. Consequently,
$((G')^*,(G_0)^*,(G'')^*)$ is an isometric cover of $G^*$,  i.e., $G^*$ is an
AMP-amalgam of $(G')^*$ and $(G'')^*$. By Proposition \ref{AMPamalgam}, $G^*$
is ample. This concludes the proof of Theorem \ref{CUOMtoAMP}.
\qed

\begin{rexample}
In Fig. \ref{fig:running_example_CUOM_completion} we present an ample completion of the running example $M$ (recall that $M$ is the tope graph of a CUOM),
obtained as in the proof of Theorem \ref{CUOMtoAMP}.
\end{rexample}

\begin{figure}[h]\centering
	\includegraphics[width=.35\textwidth]{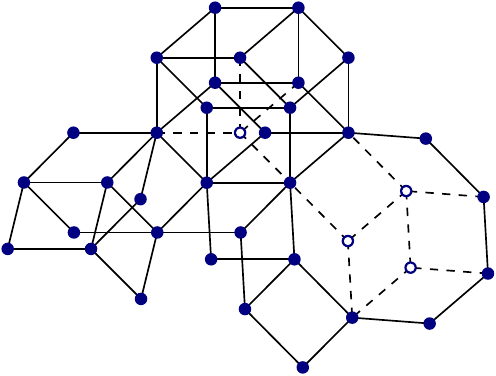}
	\caption{An ample completion of the running example $M$.}
	\label{fig:running_example_CUOM_completion}
\end{figure}

\section{Discussion}
We proved that the tope graph of every OM or CUOM of VC-dimension $d$ has an 
ample completion of the same VC-dimension.
For OMs, this result is proved in two stages: first, we complete the OM to a UOM  (using the theory of
oriented matroids) and then, recursively we complete the tope graph of the resulting UOM to an AMP.
For CUOMs, the completion is obtained by completing each facet independently  and by taking the union of
those facet completions.  Since ample set systems of VC-dimension $d$ admit 
labeled compression schemes of size $d$~\cite{MoWa} and this property is 
closed under taking subsystems, see Subsection 
\ref{subsect_sample_compression_schemes}, from Theorems \ref{OMtoAMP} and 
\ref{CUOMtoAMP} we obtain :

\begin{corollary}
	Concept classes defined by the topes of an OM or a CUOM of VC-dimension $d$ 
	admit labeled compression schemes of size $d$.
\end{corollary}

For general COMs, one can envisage the same strategy: complete each  facet of 
the tope graph
and take their union. However, if the completion
of faces is done as for OMs, 
then, as shown in the following example, a few difficulties arise.

\begin{example}
	In Fig. \ref{fig:C6P3}(a)  we present the tope graph $G$ of a COM  of
	VC-dimension 3, which is
	the Cartesian product $\C_8\product P_3$. It consists of two facets (which
	are both prisms $\C_8\product P_2$) glued together along a common face
	$\C_8$. In  Fig. \ref{fig:C6P3}(b) we complete each facet to the tope graph of a UOM. 
	However, the resulting graph is not even a partial cube. This problem
	arises for any completion of the two facets to tope graphs of UOMs.
Nevertheless, the graph has an ample completion of the same VC-dimension, see
Fig. \ref{fig:C6P3}(c).
\end{example}

\begin{figure}[h]\centering
	\includegraphics[height=5cm]{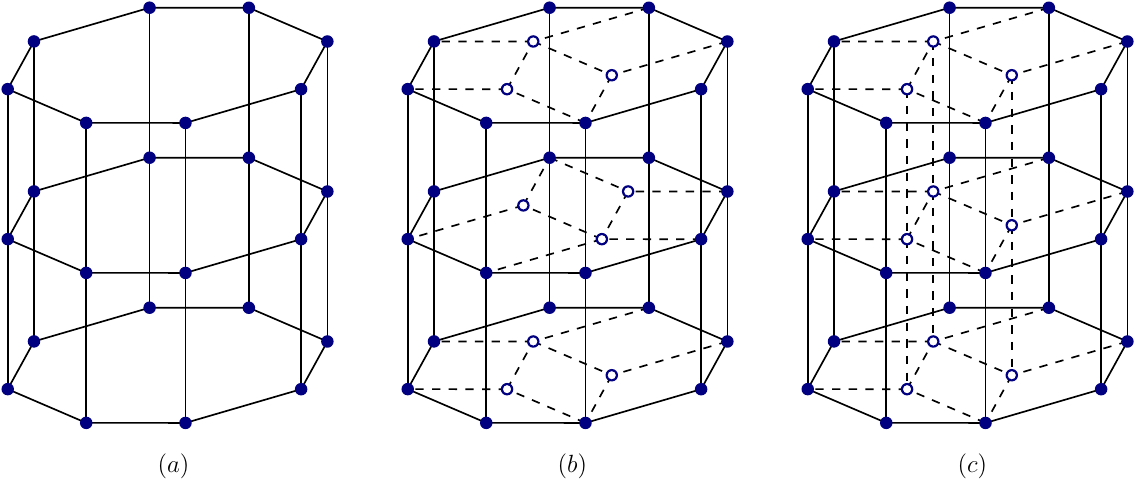}
	\caption{(a) The tope graph $G$ of a COM corresponding to $C_8\product P_3$.  (b) The
	partial completion of $G$ after completing the two facets. (c) The smallest
	ample completion of $G$.}\label{fig:C6P3}
\end{figure}

Let us discuss what we learn from the above example.
The intersections of ample set systems with cubes is ample and  all faces of 
the tope graph $G$ of
a COM are gated (and thus convex), thus if $\amp(G)$ is an ample completion of $G$ and $\F(X)$
is a face of $G$, then the intersection of $\amp(G)$ with the smallest cube $\C(X)$
containing $\F(X)$ is ample and thus is an ample completion of $\F(X)$. This explains
why we should take care of the completions of faces.

The completion of $C_8\product P_3$  from  Fig. \ref{fig:C6P3}(c) satisfies the
following \emph{parallel faces completion property}: any two parallel faces
$\F(X)$ and $\F(Y)$ of $G$ are completed in the same way, i.e., the isomorphism
between $\F(X)$ to $\F(Y)$ (given by metric projection) extends to an
isomorphism between the completions $\amp(G)\cap \C(X)$ and $\amp(G)\cap
\C(Y)$.  Since in tope graphs of COMs parallel faces are not facets, in  CUOMs they are cubes,
and we conclude that the completion of CUOMs satisfies the parallel faces
completion property. We believe, that \emph{Conjecture \ref{conjectureCOM} can be
strengthened by furthermore imposing the parallel faces completion property.}

%

\begin{example}
	\label{partialcube-completionVCdim4}
	In \cite{ChKnPh} we proved that any partial cube of VC-dimension 2 admits
	an ample completion of VC-dimension 2. The example from Fig.
	\ref{fig:contre_ex} shows that this is no longer true for partial cubes of
	VC-dimension 3. The graph is an isometric subgraph of $Q_5$, has
	VC-dimension $3$, and all its ample completions have VC-dimension at least
	$4$. There are six such subgraphs of $Q_5$ and the one in Fig.
	\ref{fig:contre_ex} is an
	isometric subgraph of all the others. On the other hand, all tope graphs of
	COMs in $Q_5$
	satisfy Conjecture~\ref{conjectureCOM}. The examples and their analysis
	have been obtained using SageMATH~\cite{sage} and the database of partial
	cubes in $Q_5$~\cite{Ma}.
\end{example}

\begin{figure}[h]\centering
	\includegraphics[height=3cm]{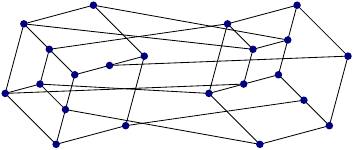}
	\caption{A partial cube of VC-dimension 3 that cannot be completed to an
	ample partial cube of the same VC-dimension.}
    \label{fig:contre_ex}
\end{figure}

The class of partial cubes, that can be completed to an ample partial cube of
VC-dimension $3$ is closed under pc-minors. What are the minimal excluded
pc-minors of this class?

\medskip\noindent
{\bf Acknowledgements.} We are grateful to the anonymous referees for useful comments and improvements. This work was supported by ANR project DISTANCIA, ANR-17-CE40-0015.
The second author was partially supported by the Spanish \emph{Ministerio de Ciencia, Innovaci\'on y  Universidades} through grants RYC-2017-22701 and PID2019-104844GB-I00.

\bibliographystyle{siamplain}

%

\begin{thebibliography}{99}
{\footnotesize
\bibitem{AlKn} M. Albenque and K. Knauer, Convexity in partial cubes: the hull number,  Discr. Math.  339 (2016), 866--876.

\bibitem{AnRoSa}  R.P. Anstee, L. R\'onyai, and A. Sali, Shattering news, Graphs and Comb. 18 (2002), 59--73.



\bibitem{BaChDrKo} H.-J. Bandelt, V. Chepoi, A. Dress, and J. Koolen, Combinatorics of lopsided sets, Eur. J. Comb. 27 (2006), 669--689.

\bibitem{BaChDrKounp} H.-J. Bandelt, V. Chepoi, A. Dress, and J. Koolen, unpublished.

\bibitem{BaChKn} H.-J. Bandelt, V. Chepoi, and K. Knauer, COMs: complexes of oriented matroids, J. Comb. Th., Ser. A 156 (2018), 195--237.


\bibitem{BjLVStWhZi} A. Bj\"orner, M. Las Vergnas, B. Sturmfels, N. White, and G. Ziegler,
      Oriented Matroids, Encyclopedia of Mathematics and its Applications, vol. 46,
      Cambridge University Press, Cambridge, 1993.



\bibitem{BlLV} R. Bland and M. Las Vergnas, Orientability of matroids, J. Comb. Th. Ser. B, 23 (1978), 94--123.

\bibitem{BoRa} B. Bollob\'as and A.J. Radcliffe, Defect Sauer results, J. Comb. Th. Ser. A 72 (1995), 189--208.


\bibitem{ChChMoWa} J. Chalopin, V. Chepoi, S. Moran, and M.K. Warmuth, Unlabeled sample compression schemes and corner
peelings for ample and maximum classes, ICALP 2019, pp.34:1-34:15,  full version: arXiv:1812.02099v1.

\bibitem{ChChMInRaVa} J. Chalopin, V. Chepoi, F. Mc Inerney, S. Ratel, and Y. Vax\`es, Distinguishing and compressing balls in graphs (in preparation). 

\bibitem{Ch_thesis} V. Chepoi, $d$-Convex sets in graphs, Dissertation, Moldova State Univ.,
Chi\c{s}in\v{a}u, 1986.

\bibitem{Ch_hamming} V. Chepoi, Isometric subgraphs of Hamming graphs and
$d$-convexity, Cybernetics 24 (1988), 6--10.


%
\bibitem{ChKnMa} V. Chepoi, K. Knauer, and T. Marc, Hypercellular graphs: 
partial cubes without $\Q_3^-$ as partial cube minor, Discr. Math. 343 (2020), 
111678.

\bibitem{ChKnPh} V. Chepoi, K. Knauer, M. Philibert, Two-dimensional partial cubes, Electron. J. Comb. 27 (2020), P3.29.



\bibitem{daS95}
I. da~Silva, Axioms for maximal vectors of an oriented
  matroid: a combinatorial characterization of the regions determined by an
  arrangement of pseudohyperplanes, Eur. J. Comb., 16 (1995), 125--145.

\bibitem{Dj} D.\v{Z}. Djokovi\'{c}, Distance--preserving subgraphs of hypercubes, J. Comb. Th. Ser. B 14 (1973), 263--267.

\bibitem{Dr} A.W.M. Dress, Towards a theory of holistic clustering, DIMACS Ser. Discr. math. Theoret. Comput.Sci., 37 Amer. Math. Soc. 1997, pp. 271--289.
%

\bibitem{DrSch} A. W. M. Dress and R. Scharlau, Gated sets in metric spaces, Aequationes Math. 34 (1987), 112--120.


\bibitem{FlWa}  S. Floyd and M.K. Warmuth, Sample compression, learnability, and the Vapnik-Chervonenkis dimension, Machine Learning 21 (1995), 269--304.

\bibitem{FoLa} J. Folkman and J. Lawrence, Oriented matroids, J. Comb. Th. Ser. B, 25 (1978), 199--236.

\bibitem{GaWe}  B. G\"artner and E. Welzl, Vapnik-Chervonenkis dimension and (pseudo-)hyperplane arrangements, Discr. Comput. Geom. 12 (1994), 399--432.

%
%
%
%
%

\bibitem{KnMa} K. Knauer and T. Marc, On tope graphs of complexes of oriented matroids, Discr. Comput. Geom. 63 (2020), 377--417.

\bibitem{KnMa1} K. Knauer and T. Marc, Corners and simpliciality in oriented matroids and partial cubes, arXiv:2002.11403, 2020.

\bibitem{La} J. Lawrence, Lopsided sets and orthant-intersection of convex sets,  Pac. J. Math.  104 (1983), 155--173.

\bibitem{LiWa}  N. Littlestone and M. Warmuth, Relating data compression and learnability, Unpublished, 1986.

%

\bibitem{Ma} T. Marc, Repository of partial cubes
\url{https://github.com/tilenmarc/partial_cubes}.

\bibitem{MoWa} S. Moran and M. K. Warmuth, Labeled compression schemes for extremal classes, ALT 2016, 34--49.

\bibitem{MoYe}  S. Moran and A. Yehudayoff, Sample compression schemes for VC classes, J. ACM 63 (2016), 1--21.


%
%
%
\bibitem{Pa} A. Pajor, Sous-espaces $\ell_1^n$ des espaces de Banach, Travaux en Cours, Hermann, Paris, 1985
%
%

\bibitem{RuRuBa} B.I.P.  Rubinstein,  J.H.  Rubinstein,  and  P.L.  Bartlett,  Bounding  embeddings  of  VC  classes  into maximum classes, in:  V. Vovk, H. Papadopoulos,
A. Gammerman (Eds.), Measures of Complexity. Festschrift for Alexey Chervonenkis, Springer, 2015, pp. 303--325.

%
\bibitem{Sauer} N. Sauer, On the density of families of sets, J. Comb. Th., Ser. A 13 (1972), 145--147.

\bibitem{Shelah}  S. Shelah, A combinatorial problem, stability and order for models and theories in infinitary languages, Pac. J. Math. 41 (1972), 247--261.

\bibitem{She} G.C. Shephard, Combinatorial properties of associated zonotopes, Can. J. Math. 26 (1974), 302--321.

\bibitem{sage}
W.\thinspace{}A. Stein et~al., \emph{{S}age {M}athematics {S}oftware
({V}ersion 8.1)}, The Sage Development Team, 2017, {\tt
http://www.sagemath.org}.

\bibitem{VaCh} V.N. Vapnik and A.Y. Chervonenkis, On the uniform convergence of relative frequencies of events to their probabilities, Theory Probab. Appl. 16 (1971), 264--280.

\bibitem{Wiedemann} D.H. Wiedemann, Hamming geometry, PhD Thesis, University of Ontario, 1986, re-typeset 2006.
}
\end{thebibliography}

\end{document}